\documentclass[sn-mathphys,Numbered]{sn-jnl}% Math and Physical Sciences Reference Style
\usepackage{graphicx}%
\usepackage{multirow}%
\usepackage{amsmath,amssymb,amsfonts}%
\usepackage{amsthm}%
\usepackage{mathrsfs}%
\usepackage[title]{appendix}%
\usepackage{xcolor}%
\usepackage{textcomp}%
\usepackage{physics}
\usepackage{float}
\usepackage{manyfoot}%
\usepackage{booktabs}%
\usepackage[ruled,vlined]{algorithm2e}
\usepackage{siunitx}
\usepackage{tikz}
\usepackage{epsfig,color}
\usepackage[mathscr]{euscript}
\usepackage{epstopdf}
\usepackage{tikz}
\usepackage{caption}
\usepackage[utf8]{inputenc}
\usepackage[english]{babel}
\usepackage{makeidx}
\usepackage{graphicx}
\usepackage{algpseudocode}
\usepackage{breqn}
\usepackage{lmodern}
\usepackage{relsize}
\usepackage{cases}
\usepackage{bm}
\usepackage{makecell}
\usepackage{tabulary}
\usepackage[]{subfig}
\usepackage{enumitem}

\theoremstyle{thmstyleone}%
\newtheorem{theorem}{Theorem}[section]% meant for sectionwise numbers
\theoremstyle{thmstyletwo}%
\newtheorem{example}{Example}%
\newtheorem{remark}{Remark}%
\newtheorem{lemma}[theorem]{Lemma}

\numberwithin{equation}{section}

\theoremstyle{thmstylethree}%

\begin{document}

\title{Error Analysis of a Fully-Discrete Implicit $\theta$-Scheme with WG-FEM for Parabolic Singularly Perturbed Boundary Turning Point Problems}


\author[1]{\fnm{Aayushman} \sur{Raina}}\email{aayushman.raina@iitg.ac.in}
\author*[2]{\fnm{Srinivasan} \sur{Natesan}}\email{natesan@iitg.ac.in}
\equalcont{The authors contributed equally to this work.}

\affil[1,2]{\orgdiv{Department of Mathematics}, \orgname{Indian Institute of Technology Guwahati}, \orgaddress{\city{Guwahati}, \postcode{781039}, \state{Assam}, \country{India}}}


\abstract{In this article, we introduce a weak Galerkin finite element method (WG-FEM) for a class of parabolic singularly perturbed boundary turning point problems (SPBTPPs). The proposed numerical scheme employs an implicit $\theta$-scheme for temporal discretization over a uniform mesh and applies WG-FEM spatial discretization on a layer-adapted Shishkin mesh.  Rigorous stability estimates are established for both the semi-discrete and fully-discrete formulations. Furthermore, we derive error estimates in the energy norm and prove that the convergence of the scheme is uniform with respect to the perturbation parameter. Numerical tests are conducted to verify the theoretical findings and illustrate the efficiency of the proposed method. In addition, the theoretical framework developed in this work lays the foundation for future extensions to higher-dimensional problems using ADI-type operator splitting WG-FEM schemes.}



%%\pacs[MSC Classification]{35A01, 65L10,
\maketitle
\section{Introduction}
Singularly perturbed problems commonly emerge across various fields of science and engineering, such as fluid dynamics, ocean circulation, heat transfer, optimal control, and financial mathematics. One of the main characteristics of these problems is the involvement of a small parameter $\varepsilon > 0$ in the highest derivative, which strongly determines the nature of the solution.
Solutions of such equations exhibits multiscale phenomena in general {\em i.e.} in certain thin region(s) of the domain the solution and/or its derivatives changes rapidly and outside of it behaves regularly and varies slowly. These regions, characterized by sharp gradients, are referred to as \textit{boundary layers or interior layers}.

Several numerical methods are available in the literature to solve singularly perturbed differential equations, for details one can refer the books \cite{fhmos_B,miller_book_1996,malley_book_1991}. Because of the formation of boundary and/or interior layers, numerical methods face major computational difficulties. Finite difference and classical finite element methods (FDMs/FEMs) yield inaccurate solutions and produce non physical oscillations on uniform meshes. Therefore, one requires to develop $\varepsilon-$uniform methods to solve such problems. One can find various works on solving singular perturbation problems (SPPs) with boundary layers. Beckett and Mackenzie in \cite{beckett2000convergence}, derive $\varepsilon-$uniform error estimates for two first-order upwind discretization of a SPP in 1D over a non-uniform mesh. Clavero and Gracia in \cite{clavero2012high} considered a 1D parabolic SPP of reaction-diffusion type. They implemented backward-Euler in time over a uniform mesh and a HODIE FDM in space over a generalized Shishkin mesh and proved first order convergence in time and almost fourth order convergence in space. In \cite{avijit2020sdfem} Avijit and Natesan studied the convergence behavior of the streamline diffusion FEM (SDFEM) applied to time-dependent singularly perturbed convection-diffusion problems in 1D. In \cite{kaland2010parabolic}, Kaland and Roos implemented Galerkin FEM with $\theta-$scheme and discontinuous Galerkin (dG) for time discretization for a parabolic  IBVP in 1D. They obtain the convergence of order almost one for both the schemes in time. Zarin and Roos in \cite{zarin2005interior} used the $h$-version of non-symmetric dG-FEM with interior penalities (NIPG) over Shishkin-type mesh to solve convection-diffusion problems in 2D. Singh and Natesan \cite{singh2022superconvergence}, considered a parabolic IBVP and discretized the temporal variable using dG on a uniform mesh and spatial variable using NIPG on a Shishkin mesh. They stuidied the superconvergence properties of the proposed scheme. The literature offers limited insight into numerical approaches for solving singularly perturbed turning point problems (SPTPPs) as such problems are harder to solve than nonturning point problems.

Natesan et al. in \cite{natesan2003parameter} introduced an $\varepsilon-$uniform method for SPTPPs exhibiting boundary layers.  Majumdar and Natesan in \cite{majumdar2022parameter} presented a numerical strategy exhibiting uniform convergence for a class of convection-diffusion SPPs in 2D exhibiting interior layer due to non-smooth convection and source term. Lin{$\ss$} in \cite{linss2003robustness} analyzed the convergence of an upwind method on Shishkin meshes for a class of semilinear SPPs with boundary turning point. In \cite{dunne2003fitted}, Dunne et al. developed an $\varepsilon-$robust numerical method for parabolic singularly perturbed boundary turning point problems (SPBTPPs) by discretizing the temporal variable using the backward-Euler scheme over a uniform grid combined with spatial discretization via the standard upwind approach on a fitted mesh. They proved an almost first order convergence of the proposed method in the maximum norm. Chen et al. in \cite{chen2008stability} considered a 1D SPBTPP and implemented a variant of SDFEM. Stability and convergence estimates of the proposed scheme were proved on Shishkin-type mesh and Bakhvalov-type mesh. Some other works on solving SPBTPPs can be found in \cite{shishkin2001grid, gupta2011layer}
and references therein.

Recently, a non-conforming FEM known as the weak Galerkin FEM (WG-FEM) is becoming popular among the researchers of SPP community. Originally developed in \cite{wang2013weak}, this approach was designed to handle second-order elliptic PDEs. The idea was to approximate the classical gradient operator with \textit{weak gradient} in the sense of distributions. The adoption of discontinuous basis functions enables the flexible and efficient construction of higher-order elements on arbitrary partitions of the computational domain, thereby accommodating complex mesh configurations with ease. WG-FEM has been widely explored in addressing problems encountered in physics and engineering such as for Stokes equations \cite{wang2016weak}, Maxwell's equations \cite{mu2015weak}, Brinkman equations \cite{mu2014stable}, etc. Within the literature, WG-FEM have been extensively applied to SPPs, encompassing both elliptic and parabolic types. In \cite{zhu2020uniformly}, Zhu and Xie proposed a WG-FEM for one-dimensional SPPs, utilizing Shishkin mesh to resolve boundary layers. In \cite{zhang2022uniform}, Zhang and Liu proposed a WG-FEM to address SPPs in 2D  over a Shishkin mesh. Tao et al. \cite{tao2024uniform} established a WG-FEM framework for solving 1D reaction-diffusion SPPs on a Shishkin mesh. Toprakseven in \cite{toprakseven2022superconvergence} analyzed the superconvergence properties of a modified WG-FEM on a Shishkin mesh. Authors in \cite{raina2025weak} proposed a WG-FEM approximation for parabolic fourth-order SPPs in 2D over a Shishkin mesh. High-dimensional parabolic equations often suffer from the curse of dimensionality, leading to significant computational challenges. To mitigate this issue researchers are employing operator-splitting techniques \cite{2011_Ahmeda_Matthies_Tobiska, 1998_Clavero_Jorge_Lisbona_Shishkin, 2012_Ganesan}. This strategy not only streamlines the solution process but also lowers computational costs and enhances flexibility when dealing with complex geometries. Despite these advantages, the integration of operator-splitting ADI methods within the WG-FEM framework for tackling high-dimensional SPPs has received limited attention in current research \cite{raina2025novel, raina2025efficient, kumar2025adi}.

The contributions of our work are enlisted as follows:
\begin{enumerate}
    \item  We propose the WG-FEM for SPBTPPs, addressing a previously unstudied class of problems.
    
    \item  We develop a numerical scheme combining an implicit $\theta$-scheme for temporal discretization on a uniform mesh with WG-FEM spatial discretization on a Shishkin mesh.
    
    \item Stability estimates for semi- and fully-discrete formulations are rigorously established.
    
    \item Error estimates in the energy norm are derived, and $\varepsilon$-uniform convergence of the proposed scheme is proved.
    
    \item  We establish theoretical foundations for future operator splitting based WG-FEM schemes for SPBTPPs in higher dimensions, leveraging the 1D bounds derived in this work.
\end{enumerate}
The organization of the paper is as follows: Section \ref{md_prb} describes the model problem, presents the solution decomposition, and establishes analytical bounds. In Section \ref{mesh}, we develop the Shishkin mesh framework for spatial discretization. Section \ref{WG} details the Weak Galerkin formulation, including both semi- and fully discrete schemes, along with their corresponding stability estimates. Section \ref{error} derives the error equations for both schemes and provides a comprehensive error analysis. Numerical experiments validating the robustness of the proposed scheme are presented in Section \ref{Numerical}. Finally, Section \ref{conclusion} concludes the article with a summary of key findings and future research directions. In the text, $C$ denotes a constant that does not depend on $\varepsilon$, $N$ or $\mathcal{M}$, and it may vary in different contexts.

\section{Model Problem}\label{md_prb}
The problem under consideration is the parabolic singularly perturbed boundary turning point problem (SPBTPP) given by:
\begin{equation}\label{spbtp_md_problem}
\left\{\begin{array}{ll}
u_{t} -\varepsilon u_{xx} -x^{q}\cdot b(x,t) u_{x} +c(x,t)u = {f}(x,t), \,\, q \geq 1, \, (x,t) \in \,\, \Omega \times (0,T],\\[6pt]
u(x,0) = u_{0}(x), \,\, x \in [0,1], \\[6pt]
u(0,t) = 0, u(1,t) = 0, \,\, t \in [0,T].
\end{array}\right.
\end{equation}
Here $0 <\varepsilon \ll 1$ is a singular perturbation parameter and $\Omega = (0,1)$. Considering that $b(x,t), c(x,t)$ and $f(x,t)$ are sufficiently smooth in their respective domains, the uniqueness of the solution to \eqref{spbtp_md_problem} is ensured. Also, we denote the convection coefficient as $x^qb(x,t) = a(x,t)$. The solution of \eqref{spbtp_md_problem} exhibits a turning point behavior when the convection coefficient becomes zero at some location within the domain $\Omega$. Specifically, if this coefficient vanishes at the domain boundary, the problem is termed a boundary turning point problem. If it vanishes at an interior point, it is referred to as an interior turning point problem. Our problem \eqref{spbtp_md_problem} contains a turning point at $x=0$, with multiplicity $q$. If $q=1$, the problem corresponds to a simple boundary turning point case and multiple boundary turning point case if $q > 1$. We also have the following assumptions
\begin{equation}\label{coeff_assumptions}
\left\{\begin{array}{ll}
b(x,t) \geq \mathfrak{b} > 0, \quad (x,t) \in [0,1] \times [0,T], \\[6pt]
c(x,t) \geq \mathfrak{c} > 0, \quad (x,t) \in [0,1]\times [0,T].
\end{array}\right.
\end{equation}
The solution of \eqref{spbtp_md_problem} under the assumptions \eqref{coeff_assumptions}, shows a parabolic boundary layer forming near the vicinity of $x=0$ with a width of $\mathcal{O}(\sqrt{\varepsilon})$.

\subsection{Solution decomposition}
For error analysis, it is essential to estimate the solution as well as its derivatives. These estimates are detailed in the following lemma.
\begin{lemma}
    The following estimates hold for the solution $u$ of \eqref{spbtp_md_problem} and its derivatives:
    \begin{equation}
    \Bigg|\dfrac{\partial^{i+j}u(x,t)}{\partial x^i \partial t^j}\Bigg| \leq C(1 + \varepsilon^{-i/2}e^{-\alpha x}),
    \end{equation}
    where the indices satisfy $0 \leq i+j \leq k$, $k$ depends on the regularity of $u$ and $\alpha = \sqrt{\mathfrak{b}/\varepsilon}$.
\end{lemma}
In order to drive higher order accurate error estimates, we provide sharper bounds on the solution of \eqref{spbtp_md_problem} in the following lemma.
\begin{lemma}\label{ch11_sol_Decomposition}
    The solution $u(x,t)$ of \eqref{spbtp_md_problem} can be expressed as the sum of a smooth part $v(x,t)$ and a layer part $w(x,t)$, given by
\begin{equation}
    u(x,t) = v(x,t) + w(x,t).
\end{equation}
Then the bounds on smooth and layer component are given by:
    \begin{eqnarray}
        \Bigg|\dfrac{\partial^{i+j}v(x,t)}{\partial x^i \partial t^j}\Bigg| &\leq& C(1 + \varepsilon^{(3-i)/2}), \\[2pt]
        \Bigg|\dfrac{\partial^{i+j}w(x,t)}{\partial x^i \partial t^j}\Bigg| &\leq& C(\varepsilon^{-i/2}e^{-\alpha x}),
    \end{eqnarray}
    respectively, where $i+j \leq k$ and $\alpha = \sqrt{\mathfrak{b}/\varepsilon}$.
\end{lemma}
For more details into these lemmas one can refer to \cite{gupta2011layer}.
\section{Mesh construction for Spatial Discretization}\label{mesh}
In this section, we will discuss the construction of a layer adapted mesh, specifically Shishkin mesh for spatial discretization.
\subsection{Construction of Shishkin mesh}
We will present here the construction of Shishkin mesh suitable for our SPBTPP \eqref{spbtp_md_problem} posed in [0,1]. We first choose the \textit{mesh transition} parameter denoted by $\tau$ as:
\begin{equation}\label{ch11_transition}
    \tau = \min\Big\{\dfrac{1}{2}, \frac{(k+1)}{\sqrt{\mathfrak{b}}}\sqrt{\varepsilon}\ln N\Big\},
\end{equation}
where $k$ corresponds to the order of polynomials chosen for the approximation space. Using the transition parameter $\tau$ we divide our computational domain $\overline{\Omega}$ into $[0,\tau]$ and $[\tau,1]$ (See Figure \ref{ch11_mesh_fig}). 

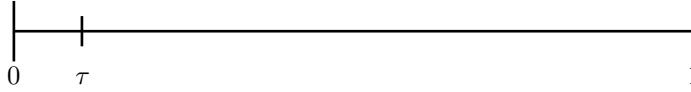
\begin{figure}[H]
\centering
\begin{tikzpicture}
\draw[line width=1pt] (0,0) -- (9,0);
\draw[line width=1pt] (0,-0.4) -- (0,0.4);
\draw[line width=1pt] (0.9,-0.2) -- (0.9,0.2);
\draw[line width=1pt] (9,-0.4) -- (9,0.4);

\node[label={0}]  at (0,-0.95) {};
\node[label=]  at (7.5,-0.9) {};
\node[label={1}]  at (9,-0.95) {};
\node[label=$\tau$]  at (0.91,-0.95) {};
\node[label=]  at (4.2,-0.8) {};
\node[label=]  at (8.30,-0.1) {};
\end{tikzpicture}
\caption{Subdivision of the computational domain using the parameter $\tau$.}\label{ch11_mesh_fig}
\end{figure}
Then each sub-domain is subsequently split into $N/2$ equally sized sub-intervals. The mesh nodes are determined as
\begin{equation}\label{ch11_xmesh}
x_{i} =
\left\{\begin{array}{ll}
\dfrac{2i\tau}{N}, \quad &i = 0,\ldots,\frac{N}{2},\\[6pt]
\tau + \dfrac{2(i-\frac{N}{2})(1-\tau)}{N}, \quad &i = \frac{N}{2}+1,\ldots,N,
\end{array}\right.
\end{equation}
and the step size as
\begin{equation}\label{ch11_xmesh_stepsize}
h_{i} =
\left\{\begin{array}{ll}
\dfrac{2\tau}{N}, \quad &i = 0,\ldots,\frac{N}{2},\\[6pt]
\dfrac{2(1-\tau)}{N}, \quad &i = \frac{N}{2}+1,\ldots,N.
\end{array}\right.
\end{equation}
The step size $h_i = x_{i+1} - x_{i}$ has the following order
\begin{eqnarray*}
    h_i &=& \mathcal{O}(\sqrt{\varepsilon}N^{-1}\ln N), \quad i \in \{0,\ldots, N/2\}, \\[2pt]
    h_i &=& \mathcal{O}(N^{-1}), \quad i \in \{N/2+1,\ldots, N\}.
\end{eqnarray*}

\section{WG-FEM for Spatial Discretization}\label{WG}
Let $M_{x_i} = [x_{i-1}, x_i]$ represent a mesh element. The domain $\Omega$ is divided into elements $M_{x_i}$, forming the partition $\mathcal{T}_N = \{ M_{x_i} \mid i=1,\ldots,N \}$. Let $h = \max\{h_i\}$ denote the mesh size, where $h_i$ is the length of the local element $M_{x_i}$ for $i = 1, 2, \ldots, N$.

A \textit{weak function} defined on the element $M_{x_i}$ is given by $u_h = \{ u_0, u_b \}$, where $u_0 \in L^2(M_{x_i})$ and $u_b \in L^\infty(\partial M_{x_i})$. The boundary of the element is denoted by $\partial M_{x_i} = \{ x_{i-1}, x_i \}$. In this formulation, $u_0$ represents the function’s value inside the element, while $u_b$ corresponds to its values at the endpoints.

Specifically, for each element $M_{x_i} \in \mathcal{T}_N$, the weak function $u_h$ is defined as:
\[
u_h =
\begin{cases}
u_0, & \text{in } M_{x_i}, \\
u_b, & \text{on } \partial M_{x_i}.
\end{cases}
\]
We now introduce the \textit{discrete weak function space} associated with each element $M_{x_i}$ as follows:
\begin{equation}
\mathscr{W}(M_{x_i}, k) = \left\{ u_h = \{ u_0, u_b \} \, \middle| \, u_0|_{M_{x_i}} \in \mathbb{P}_k(M_{x_i}), \, u_b|_{\partial M_{x_i}} \in \mathbb{P}_0(\partial M_{x_i}) \right\},
\end{equation}
where $\mathbb{P}_k(M_{x_i})$ denotes the space of polynomials of degree at most $k \in \mathbb{Z}^+$ defined on $M_{x_i}$, and $\mathbb{P}_0(\partial M_{x_i})$ represents the space of functions that are piecewise constant on the boundary $\partial M_{x_i}$.

We define the global weak Galerkin space $\mathscr{S}_h$ by assembling the local spaces $\mathscr{W}(M_{x_i}, k)$ and enforcing continuity at the interior mesh nodes. Specifically, we set
\begin{equation}
    \mathscr{S}_h = \mathscr{S}_h(\Omega) = \left\{ u_h = \{ u_0, u_b \} \, : \, \{ u_0, u_b \}|_{M_{x_i}} \in \mathscr{W}(M_{x_i}, k), \; \forall M_{x_i} \in \mathcal{T}_N \right\}.
\end{equation}

We further define a subspace $\mathscr{S}_h^0 \subset \mathscr{S}_h$ consisting of weak functions whose boundary components vanish at the endpoints of the domain. That is,
\begin{equation} \label{ch11_zero}
    \mathscr{S}_h^0 = \left\{ u_h \in \mathscr{S}_h \, : \, u_b(0) = u_b(1) = 0 \right\}.
\end{equation}

Equation~\eqref{ch11_zero} shows that boundary conditions are strongly imposed in the proposed WG-FEM.

\subsection{Weak Derivative}
For a weak function $u_h \in \mathscr{W}(M_{x_i}, k)$, the weak derivative $d_{w, M_{x_i}} u_h \in \mathbb{P}_{k-1}(M_{x_i})$ is the unique polynomial that satisfies
\begin{equation} \label{weakderivative}
   (d_{w, M_{x_i}} u_h, \phi)_{M_{x_i}} = - (u_0, \phi')_{M_{x_i}} + \langle u_b, \phi n \rangle_{\partial M_{x_i}}, \quad \forall \phi \in \mathbb{P}_{k-1}(M_{x_i}),
\end{equation}
where the inner product and boundary pairing are defined by
\[
(\kappa_1, \kappa_2)_I := \int_I \kappa_1(x) \kappa_2(x) \, dx, \quad \text{and} \quad \langle \kappa_1, \kappa_2 n \rangle_{\partial I} := \kappa_1(x_i)\kappa_2(x_i) - \kappa_1(x_{i-1})\kappa_2(x_{i-1}),
\]
for $I = (x_{i-1}, x_i)$.

To approximate the convection term, we define the weak convection derivative as follows:  
For any weak function $u_h \in \mathscr{W}(M_{x_i}, k)$, the weak convection derivative $d_{w, M_{x_i}}^{a} u_h \in \mathbb{P}_k(M_{x_i})$ is the unique polynomial satisfying
\begin{equation} \label{convectionderivative}
(d_{w, M_{x_i}}^{a} u_h, \phi)_{M_{x_i}} = - (u_0, (a\phi)')_{M_{x_i}} + \langle u_b, a\phi n \rangle_{\partial M_{x_i}}, \quad \forall \phi \in \mathbb{P}_k(M_{x_i}).
\end{equation}

We compute the weak derivatives $d_w$ and $d_w^{a}$ elementwise within the global WG finite element space $\mathscr{S}_h$ on each interval $M_{x_i}$ for $i = 1, \dots, N$. Specifically,
\begin{equation*}
   (d_w u_h)|_{M_{x_i}} = d_{w, M_{x_i}}(u_h|_{M_{x_i}}), \quad (d_w^{a} u_h)|_{M_{x_i}} = d_{w, M_{x_i}}^{a}(u_h|_{M_{x_i}}), \quad \forall u_h \in \mathscr{S}_h.
\end{equation*}

\subsection{Weak Formulation and Well-Posedness}
In this subsection, we present both the semi-discrete and fully-discrete WG-FEM formulations for problem~\eqref{spbtp_md_problem}, employing implicit $\theta-$scheme for time discretization.

Let $u_h \in \mathscr{S}_{h}^{0}$ and $v_h \in \mathscr{S}_{h}^{0}$ be given. The \textit{semi-discrete} formulation of~\eqref{spbtp_md_problem} is: find $u_h \in \mathscr{S}_{h}^{0}$ such that
\begin{equation}\label{ch11_semi_discrete}
    ((u_{0})_t, v_0) + A(u_h, v_h) = (f, v_0),
\end{equation}
with the bilinear form $A(\cdot,\cdot)$ specified by
\begin{eqnarray}\label{A_bilinear}
 A(u_h, v_h) &=& \sum_{M_{x_i} \in \mathcal{T}_{N}} \varepsilon (d_{w} u_h, d_{w} v_h)_{M_{x_i}} 
 - \sum_{M_{x_i} \in \mathcal{T}_{N}} (d_{w}^{a} u_h, v_0)_{M_{x_i}} 
 + \sum_{M_{x_i} \in \mathcal{T}_{N}} (c u_0, v_0)_{M_{x_i}} 
  \nonumber \\[2pt]
 && + s_d(u_h, v_h) + s_c(u_h, v_h),
\end{eqnarray}
and the stabilizer terms corresponding to diffusion and convection are given by
\begin{eqnarray}
   s_d(u_h, v_h) &=& \sum_{i=1}^{N} \langle \vartheta^i (u_0 - u_b), v_0 - v_b \rangle_{\partial M_{x_i}}, \\
   s_c(u_h, v_h) &=& \sum_{i=1}^{N} \langle a n (u_0 - u_b), v_0 - v_b \rangle_{\partial_+ M_{x_i}},
\end{eqnarray}
with
\[
\partial_+ M_{x_i} = \{x \in \partial M_{x_i} : a(x) n_{M_{x_i}}(x) \ge 0\},
\]
and $\vartheta^i$ defined as a penalty parameter:
\begin{align} \label{vartheta}
 \vartheta^i = \begin{cases}
 N(\ln N)^{-1}, & i \in \{1, 2, \dots, N/2\}, \\
 1, & i \in \{N/2+1, \ldots, N\}.
 \end{cases}
\end{align}
Now, we define the energy norm as for any $v_h \in \mathscr{S}_h$
\begin{equation}\label{energynormx}
     \||{v}_h\||^2 \,\, := \, \left(\vert {v}_h \vert\strut_{w}\right)^2+ \Vert \sqrt{\gamma} {v}_0\Vert^2 +\left(\vert {v}_h\vert\strut_{C}\right)^2,
\end{equation}
where
\begin{eqnarray*}
   \left(\vert {v}_h \vert\strut_{w}\right)^2 &:=& \varepsilon\Vert d_w {v}_h\Vert^2 + s_d({v}_h, {v}_h), \\[4pt]
   c(x) &+& \frac{a'(x)}{2} \geq \gamma > 0, \\[4pt]
   \left(\vert {v}_h \vert\strut_{C}\right)^2 &:=& \sum _{i=1}^{ N}d_i^*|\sqrt{a}({v}_0 - {v}_b)|^2(x_i^-),
\end{eqnarray*}
with 
$$\begin{aligned} d_i^*=\left\{ \begin{array}{ll} \frac{3}{2},&{}\quad i =  N,\\[6pt]
1,&{}\quad i=1,\ldots ,  N-1. \end{array}\right. \end{aligned}$$
Further, we measure the errors in the discrete $H^{1}-$energy norm defined as
\begin{equation}\label{energynormxH1}
    \||{v}_h\||\strut_{\mathcal{E}}^2 \,\, := \, \left(\vert {v}_h\vert\strut_{1}\right)^2+ \Vert \sqrt{\gamma} {v}_0\Vert^2 + \left(\vert {v}_h\vert\strut_{C}\right)^2,
\end{equation}
with
\begin{equation*}
\left(\vert {v}_h\vert\strut_{1}\right)^2 \,\, := \, \varepsilon\Vert D {v}_0\Vert^2 + s_d({v}_h, {v}_h),
\end{equation*}
where $D{v}$ refers to the classical derivative of $v$ taken with respect to the space variable.

\begin{remark}
    The bilinear form $A(\cdot , \cdot)$ as defined in \eqref{A_bilinear} is coercive on $\mathscr{S}_{h}^{0}$ with respect to the norm $\||\cdot\||$ which means $A(v_h,v_h) \geq \||v_h\||^2, \forall v_h \in \mathscr{S}_{h}^{0} $. It is a straightforward result and can be obtained by following \cite[Lemma 3.2]{zhu2020uniformly}.
\end{remark}

\begin{lemma}
    Given any $v_h \in \mathscr{S}_{h}^0$, there exists positive constants $\delta_1$ and $\delta_2$ for which the following inequality holds
    \begin{equation}
        \delta_1\||v_h\|| \,\, \leq \,\, \||v_h\||\strut_{\mathcal{E}} \leq \delta_2\||v_h\||.
    \end{equation}
\end{lemma}
\begin{lemma}[Semi-discrete Stability]
    There is a positive constant $C$, such that
\begin{equation}\label{ch11_semidiscrete_stability}
    \|u_{h}(t)\| \,\, \leq \,\, C\Big(\|u_h(0)\| + \int_{0}^{t}f(s)ds\Big).
\end{equation}
\end{lemma}
\begin{proof}
Substitute $v_h = u_h(t)$ into the semi-discrete formulation \eqref{ch11_semi_discrete}:
\begin{eqnarray*}
    (u_{h,t}(t), u_{h}(t)) + A(u_{h},u_{h}) = (f,u_{h}).
\end{eqnarray*}
Using the identity $(u_{h,t}(t), u_h(t)) = \dfrac{1}{2} \dfrac{d}{dt} \|u_h\|^2$, we obtain
\begin{eqnarray*}
    \dfrac{1}{2}\dfrac{d}{dt}\|u_{h}\|^{2} + A(u_{h},u_{h}) = (f,u_{h}).
\end{eqnarray*}
By the positivity of the bilinear form $A(\cdot,\cdot)$ ($A(u_h,u_h) \geq 0$), it follows that
\begin{eqnarray*}
    \dfrac{1}{2}\dfrac{d}{dt}\|u_{h}\|^{2} \,\, \leq \,\, (f,u_{h}).
\end{eqnarray*}
Then we have
\begin{eqnarray*}
(f, u_h)  &\leq& \|f\| \|u_h\| \quad (\mbox{by the Cauchy-Schwarz inequality})\\[6pt]
&\leq& \frac{1}{2} \|f\|^2 + \frac{1}{2} \|u_h\|^2 \quad(\mbox{applying Young's inequality}).
\end{eqnarray*}
This leads us to
\begin{eqnarray*}
\frac{d}{dt} \|u_h\|^2 \,\, \leq \,\, \|f\|^2 + \|u_h\|^2.
\end{eqnarray*}
Integrating over $[0,t]$
\begin{equation*}
    \|u_{h}(t)\|^{2} \,\, \leq \,\, \|u_{h}(0)\|^{2} + \int_{0}^{t} \|f\|^{2}\, ds + \int_{0}^{t} \|u_{h}\|^{2}\, ds.
\end{equation*}
Finally, by using Gr\"onwall's inequality we have
\begin{equation*}
\|u_{h}(t)\|^2 \, \leq C\left(\|u_{h}(0)\|^2 + \int_{0}^{t}\|f(s)\|^2\,ds\right), \, \forall t \in (0, T].
\end{equation*}
This completes the proof. 
\end{proof}
Using the stability estimate \eqref{ch11_semidiscrete_stability}, one can prove that the semi-discrete formulation \eqref{ch11_semi_discrete} admits atmost one solution and hence well-posed. 

\subsubsection{Fully discrete formulation}
We will introduce an evenly spaced discretization for time domain $(0,T]$. Let $\mathcal{M}$ be the number of partitions over the time interval. Choosing the step size $\zeta = T/\mathcal{M}$, we can define the mesh as
\begin{equation}\label{ch11_time_mesh}
    \mathscr{T}^{\mathcal{M}} = \{t_j: t_j = j\zeta, \,\, j = 0,1,\ldots, \mathcal{M} \}.
\end{equation}
Discretizing the semi-discrete formulation \eqref{ch11_semi_discrete} in time via the implicit $\theta$-scheme yields the fully-discrete problem:

Find $u_{h}^{n} \in \mathscr{S}_{h}^{0}$ such that
\begin{equation}\label{fully_discrete_form}
\Bigg\{\begin{array}{ll} (\overline{\partial} u_h^{n}, v_0) + A(\theta u_{h}^{n} + (1-\theta)u_{h}^{n-1}, v_{h}) = (\theta f(t_n) + (1-\theta)f(t_{n-1}), v_{0}), \quad \forall v_h \in \mathscr{S}_{h}^{0},\\[6pt]
u_{h}^{0} = \mathscr{I}u_{0},
\end{array}.
\end{equation}
where $\theta \in [\frac{1}{2}, 1]$ and $\overline{\partial} u_h^{n} := \dfrac{u_{h}^{n} - u_{h}^{n-1}}{\zeta}$.
\begin{lemma}[Fully-discrete Stability]
The fully-discrete solution $u_{h}^{n}$ satisfies the following stability estimate:
\begin{equation}
    \|u_{h}^{n}\| \,\, \leq \,\, \|u_{h}^{0}\| \,+\, C\sup_{t\in [0,T]}\|f(t)\|,
\end{equation}
    where $\|f(t)\|$ remains bounded over $[0,T]$ and $C >0$ is a constant with exponential dependence on $T$.
\end{lemma}
\begin{proof}
Put $v_{h} = \theta u_{h}^{n} + (1-\theta)u_{h}^{n-1}$ in \eqref{fully_discrete_form}, we have
\begin{eqnarray*}
    (\overline{\partial} u_h^{n}, \theta u_{h}^{n} + (1-\theta)u_{h}^{n-1}) + \||\theta u_{h}^{n} + (1-\theta)u_{h}^{n-1}\||^2 \leq (\theta f(t_n) + (1-\theta)f(t_{n-1}), \theta u_{h}^{n} + (1-\theta)u_{h}^{n-1}).
\end{eqnarray*}
Using the identity $\theta u_{h}^{n} + (1-\theta)u_{h}^{n-1} = (\theta - \frac{1}{2})(u_{h}^{n} - u_{h}^{n-1}) + \frac{1}{2}(u_{h}^{n} + u_{h}^{n-1})$, we obtain
\begin{eqnarray*}
    \Big(u_{h}^{n} - u_{h}^{n-1}, \theta u_{h}^{n} + (1-\theta)u_{h}^{n-1}\Big) &=& \Big(u_{h}^{n} - u_{h}^{n-1}, (\theta - \frac{1}{2})(u_{h}^{n} - u_{h}^{n-1}) + \frac{1}{2}(u_{h}^{n} + u_{h}^{n-1})\Big) \\[2pt]
    &=& (\theta - \frac{1}{2})\|u_{h}^{n} - u_{h}^{n-1}\|^2 + \frac{1}{2}(u_{h}^{n} - u_{h}^{n-1}, u_{h}^{n} + u_{h}^{n-1}) \\[2pt]
    &=& (\theta - \frac{1}{2})\|u_{h}^{n} - u_{h}^{n-1}\|^2 + \frac{1}{2}\|u_{h}^{n}\|^2 - \frac{1}{2}\|u_{h}^{n-1}\|^2.
\end{eqnarray*}
This yields
\begin{eqnarray*}
    \frac{1}{2}\|u_{h}^{n}\|^2 - \frac{1}{2}\|u_{h}^{n-1}\|^2 &+& (\theta - \frac{1}{2})\|u_{h}^{n} - u_{h}^{n-1}\|^2 + \zeta\||\theta u_{h}^{n} + (1-\theta)u_{h}^{n-1}\||^2 \\[2pt] 
    &\leq& \zeta (\theta f(t_n) + (1-\theta)f(t_{n-1}), \theta u_{h}^{n} + (1-\theta)u_{h}^{n-1}).
\end{eqnarray*}
By first employing the Cauchy-Schwarz inequality and then using Young's inequality, we get
\begin{eqnarray*}
    \frac{1}{2}\|u_{h}^{n}\|^2 - \frac{1}{2}\|u_{h}^{n-1}\|^2 &+& (\theta - \frac{1}{2})\|u_{h}^{n} - u_{h}^{n-1}\|^2 + \zeta\||\theta u_{h}^{n} + (1-\theta)u_{h}^{n-1}\||^2 \\[2pt] 
    &\leq& \zeta \|\theta f(t_n) + (1-\theta)f(t_{n-1})\|^2 + \frac{\zeta}{4}\||\theta u_{h}^{n} + (1-\theta)u_{h}^{n-1}\||^2.
\end{eqnarray*}
Thus, we have
\begin{eqnarray*}
 \|u_{h}^{n}\|^2  &\leq&  \|u_{h}^{n-1}\|^2 +  2\zeta \|\theta f(t_n) + (1-\theta)f(t_{n-1})\|^2 \\[2pt]
 &\leq& \|u_{h}^{0}\|^2 + C\zeta\sum_{j=1}^{n}\|\theta f(t_j) + (1-\theta)f(t_{j-1})\|^2 \\[2pt]
 &\leq& \|u_{h}^{0}\|^2 + C\zeta\sum_{j=1}^{n}(\theta \|f(t_j)\| + (1-\theta)\|f(t_{j-1})\|)^2 \\[2pt]
 &\leq& \|u_{h}^{0}\|^2 + C\zeta\sum_{j=1}^{n}(\theta \sup_{t\in [0,T]}\|f(t)\| + (1-\theta)\sup_{t\in [0,T]}\|f(t)\|)^2 \\[2pt]
&\leq& \|u_{h}^{0}\|^2 + CT\sup_{t\in [0,T]}\|f(t)\|^2 \\[2pt]
&\leq& \|u_{h}^{0}\|^2 + C\sup_{t\in [0,T]}\|f(t)\|^2.
\end{eqnarray*}
\end{proof}

\section{Error Estimates for Semi- and Fully-Discrete WG-FEM}\label{error}
In this section, we shall present the error estimates of the proposed semi- and fully-discrete schemes. We first introduce interpolation operator as follows:
Let the collection of $k+1$ nodal functions on ${M}_{x_i} \in \mathcal{T}_{N}$ be given by
\[
N_{0}(s) = s(x_{i-1}),\quad  N_{k}(s) = s(x_{i}),
\]
\[
N_{l}(s) = h_{i}^{-l}\int_{{M}_{x_i}}^{} (x-x_{i-1})^{k-1}\,s(x)\,dx,\,\, l = 1,\ldots,k-1.
\]
Using these nodal functions, we proceed to construct a local interpolation operator $\mathscr{I} : H^{1}(M_{x_i}) \to \mathbb{P}_{k}(M_{x_i})$ by
\begin{equation}\label{ch11_inter_op}
N_{l}(\mathscr{I}s - s) = 0,\,\, l = 0,\ldots,k.
\end{equation}
Since $\mathscr{I}s\vert_{{M}_{x_i}}$ is continuous, $\{\mathscr{I}s\vert_{M_{x_i}}, \mathscr{I}s\vert_{\partial M_{x_i}}\}$ is in our weak finite element space $\mathscr{S}_{h}$. 

\begin{lemma}[\cite{tobiska2006analysis}]\label{interp_bounds}
    For any element $M_{x_i} \in \mathcal{T}_{N}$ and $s \in H^{k+1}(M_{x_i})$, the interpolant $\mathscr{I}s$ constructed in \eqref{ch11_inter_op} holds
    \begin{eqnarray*}
        |s - \mathscr{I}s|\strut_{l,M_{x_i}} &\leq& Ch_{i}^{k+1-l}|s|\strut_{k+1, M_{x_i}}, \quad l = 0,1,\ldots,k+1, \\[2pt]
        \|s - \mathscr{I}s\|\strut_{ \infty, M_{x_i}} &\leq& Ch_{i}^{k+1}|s|\strut_{k+1, \infty, M_{x_i}}.
    \end{eqnarray*}
\end{lemma}
Using Lemma \ref{ch11_sol_Decomposition} in conjunction with Lemma \ref{interp_bounds}, we obtain the following bounds for the interpolation error.
\begin{lemma}\label{ch11interp_estimates}
    Let $u = v + w$ be the solution decomposition into smooth and layer component, respectively, $\mathscr{I}v$ and $\mathscr{I}w$ their interpolations as defined in \eqref{ch11_inter_op}. Then $\mathscr{I}u = \mathscr{I}v + \mathscr{I}w$, and we have
    \begin{eqnarray}
        \|u - \mathscr{I}u\|\strut_{L^{\infty}(\Omega_1)} &\leq& C(N^{-1}\ln N)^{k+1}, \\[2pt]
        \|u - \mathscr{I}u\|\strut_{L^{\infty}(\Omega_2)} &\leq& C(N^{-(k+1)}), \\[2pt]
        \|w - \mathscr{I}w\|\strut_{L^{2}(\Omega_1)} &\leq& \varepsilon^{1/4}(N^{-1}\ln N)^{k+1}, \\[2pt]
        \|(v - \mathscr{I}v)^{(m)}\|\strut_{L^{2}(\Omega)} &\leq& CN^{m - (k+1)}, \,\, m=0,1,\dots,k, \\[2pt]
        \|\mathscr{I}w\|\strut_{L^{2}(\Omega_2)} + N^{-1}\|(\mathscr{I}w)^{\prime}\|\strut_{L^{2}(\Omega_2)} &\leq& C(N^{-(k + 3/2)} + \varepsilon^{1/4}N^{-(k+1)}), \\[2pt]
        \|w\|\strut_{L^{\infty}(\Omega_2)} + \varepsilon^{-1/4}\|w\|\strut_{L^{2}(\Omega_2)} &\leq& CN^{-(k+1)}, \\[2pt]
        \|w^{\prime}\|\strut_{L^{2}(\Omega_2)} &\leq& C\varepsilon^{-1/4}N^{-(k+1)}, 
    \end{eqnarray}
    where $ \Omega_1 = [0, \tau], \Omega_2 = [\tau, 1]$ and $\sqrt{\varepsilon}\ln N \leq \sqrt{\mathfrak{b}}/2(k+1)$ is assumed.
\end{lemma}
\begin{proof}
    The bounds can be established by arguments identical to those in Lemma 12 of \cite{tobiska2006analysis}, and thus the details are omitted.
\end{proof}

\begin{lemma}\label{interp_deri_bound}
    Assuming $u \in H^{k+1}(\Omega)$ and the conditions of Lemma \ref{ch11interp_estimates}, we have
    \begin{eqnarray*}
        \|(w - \mathscr{I}w)^{(m)}\|\strut_{L^{2}(\Omega_2)} &\leq& C\varepsilon^{\frac{1}{4} - \frac{m}{2}}N^{-(k+1)}, \\[2pt]
        \|(w - \mathscr{I}w)^{(m)}\|\strut_{L^{2}(\Omega_1)} &\leq& C\varepsilon^{\frac{1}{4} - \frac{m}{2}}(N^{-1}\ln N)^{k+1-m},
    \end{eqnarray*}
    for $m = 1, 2$.
\end{lemma}

\begin{lemma}\label{deri_interp_bound}
    For $u \in H^{k+1}(\Omega)$, we establish the bound
    \begin{equation}
        \Bigg(\sum_{i=1}^{N} \frac{\varepsilon}{\vartheta^i} \|(u - \mathscr{I}u)^{\prime}\|\strut_{L^{2}(\partial M_{x_i})}^2\Bigg)^{1/2} \,\, \leq \,\, CN^{-k}\ln^{k} N,
    \end{equation}
    where $\vartheta^{i}$ is defined in \eqref{vartheta}.
\end{lemma}
\begin{proof}
    Consider $\chi_{v} = v - \mathscr{I}v$ and $\chi_{w} = w - \mathscr{I}w$, be the interpolation error in smooth and layer region, respectively. Then we can write the total interpolation error as $\chi_{u} = u - \mathscr{I}u$ or $\chi_{u} = (v + w) - \mathscr{I}(v + w) = \chi_{v} + \chi_{w}$. 

    One can see that by using the triangle inequality, we get
    \begin{equation}
        \sum_{i=1}^{N} \frac{\varepsilon}{\vartheta^i} \|\chi_{u}^{\prime}\|\strut_{L^{2}(\partial M_{x_i})}^2 \leq 2\sum_{i=1}^{N} \Big(\frac{\varepsilon}{\vartheta^i} \|\chi_{v}^{\prime}\|\strut_{L^{2}(\partial M_{x_i})}^2 \,\, + \,\, \|\chi_{w}^{\prime}\|\strut_{L^{2}(\partial M_{x_i})}^2\Big).
    \end{equation}
    Now employing the trace inequality, we get
    \begin{eqnarray}
        \|\chi_{v}^{\prime}\|\strut_{L^{2}(\partial M_{x_i})}^2 &\leq& C(h_{i}^{-1}\|\chi_{v}^{\prime}\|\strut_{L^2(M_{x_i})}^2 + \|\chi_{v}^{\prime}\|\strut_{L^2(M_{x_i})}\|\chi_{v}^{\prime\prime}\|\strut_{L^2(M_{x_i})}), \label{v_interp}\\[2pt]
        \|\chi_{w}^{\prime}\|\strut_{L^{2}(\partial M_{x_i})}^2 &\leq& C(h_{i}^{-1}\|\chi_{w}^{\prime}\|\strut_{L^2(M_{x_i})}^2 + \|\chi_{w}^{\prime}\|\strut_{L^2(M_{x_i})}\|\chi_{w}^{\prime\prime}\|\strut_{L^2(M_{x_i})}). \label{w_interp}
    \end{eqnarray}
    Now, we will bound \eqref{v_interp} first. Using Lemma \ref{interp_bounds}, we can get
    \begin{eqnarray}
        \sum_{i=1}^{N} \frac{\varepsilon}{\vartheta^i} \|\chi_{v}^{\prime}\|\strut_{L^{2}(\partial M_{x_i})}^2 &\leq& C\sum_{i=1}^{N}\frac{\varepsilon}{\vartheta^i}(h_{i}^{-1}\|\chi_{v}^{\prime}\|\strut_{L^2(M_{x_i})}^2 + \|\chi_{v}^{\prime}\|\strut_{L^2(M_{x_i})}\|\chi_{v}^{\prime\prime}\|\strut_{L^2(M_{x_i})}) \nonumber \\[2pt]
        &\leq& C\Big(\sum_{i=1}^{N/2}\frac{\varepsilon}{\vartheta^i}(h_{i}^{-1}\|\chi_{v}^{\prime}\|\strut_{L^2(M_{x_i})}^2 + \|\chi_{v}^{\prime}\|\strut_{L^2(M_{x_i})}\|\chi_{v}^{\prime\prime}\|\strut_{L^2(M_{x_i})}) \nonumber \\[2pt] 
        && + \sum_{i=N/2+1}^{N}\frac{\varepsilon}{\vartheta^i}(h_{i}^{-1}\|\chi_{v}^{\prime}\|\strut_{L^2(M_{x_i})}^2 + \|\chi_{v}^{\prime}\|\strut_{L^2(M_{x_i})}\|\chi_{v}^{\prime\prime}\|\strut_{L^2(M_{x_i})})\Big) \nonumber \\[2pt]
        &\leq& C\Big((\sqrt{\varepsilon}\|\chi_{v}^{\prime}\|\strut_{L^2(\Omega_1)}^2 + \sqrt{\varepsilon} N^{-1}\ln N \|\chi_{v}^{\prime}\|\strut_{L^2(\Omega_1)}^2\|\chi_{v}^{\prime\prime}\|\strut_{L^2(\Omega_1)}) \nonumber \\[2pt]
        && + (\varepsilon N \|\chi_{v}^{\prime}\|\strut_{L^2(\Omega_2)}^2 + \varepsilon \|\chi_{v}^{\prime}\|\strut_{L^2(\Omega_2)}^2\|\chi_{v}^{\prime\prime}\|\strut_{L^2(\Omega_2)})\Big) \nonumber \\[2pt]
        &\leq& C \sqrt{\varepsilon}N^{-2k},
    \end{eqnarray}
    where we used $\sqrt{\varepsilon}N < 1$ and $\sqrt{\varepsilon}\ln N < 1$.

    Now, we will bound \eqref{w_interp} on the same lines as we bounded \eqref{v_interp}, using Lemma \ref{interp_bounds} as follows
    \begin{eqnarray}
         \sum_{i=1}^{N} \frac{\varepsilon}{\vartheta^i} \|\chi_{w}^{\prime}\|\strut_{L^{2}(\partial M_{x_i})}^2 &\leq& C\sum_{i=1}^{N}\frac{\varepsilon}{\vartheta^i}(h_{i}^{-1}\|\chi_{w}^{\prime}\|\strut_{L^2(M_{x_i})}^2 + \|\chi_{w}^{\prime}\|\strut_{L^2(M_{x_i})}\|\chi_{w}^{\prime\prime}\|\strut_{L^2(M_{x_i})}) \nonumber \\[2pt]
        &\leq& C\Big(\sum_{i=1}^{N/2}\frac{\varepsilon}{\vartheta^i}(h_{i}^{-1}\|\chi_{w}^{\prime}\|\strut_{L^2(M_{x_i})}^2 + \|\chi_{w}^{\prime}\|\strut_{L^2(M_{x_i})}\|\chi_{w}^{\prime\prime}\|\strut_{L^2(M_{x_i})}) \nonumber \\[2pt] 
        && + \sum_{i=N/2+1}^{N}\frac{\varepsilon}{\vartheta^i}(h_{i}^{-1}\|\chi_{w}^{\prime}\|\strut_{L^2(M_{x_i})}^2 + \|\chi_{w}^{\prime}\|\strut_{L^2(M_{x_i})}\|\chi_{w}^{\prime\prime}\|\strut_{L^2(M_{x_i})})\Big) \nonumber \\[2pt]
        &\leq& C\Big((\sqrt{\varepsilon}\|\chi_{w}^{\prime}\|\strut_{L^2(\Omega_1)}^2 + \sqrt{\varepsilon} N^{-1}\ln N \|\chi_{w}^{\prime}\|\strut_{L^2(\Omega_1)}^2\|\chi_{w}^{\prime\prime}\|\strut_{L^2(\Omega_1)}) \nonumber \\[2pt]
        && + (\varepsilon N \|\chi_{w}^{\prime}\|\strut_{L^2(\Omega_2)}^2 + \varepsilon \|\chi_{w}^{\prime}\|\strut_{L^2(\Omega_2)}^2\|\chi_{w}^{\prime\prime}\|\strut_{L^2(\Omega_2)})\Big) \nonumber \\[2pt]
        &\leq& C \Big(\sqrt{\varepsilon}N^{-(2k+1)} + N^{-2(k+1)} + (N^{-1}\ln N)^{2k}\Big),
    \end{eqnarray}
    where we used Lemma \ref{interp_deri_bound}. Therefore, we obtain
    \begin{equation*}
         \Bigg(\sum_{i=1}^{N} \frac{\varepsilon}{\vartheta^i} \|(u - \mathscr{I}u)^{\prime}\|\strut_{L^{2}(\partial M_{x_i})}^2\Bigg)^{1/2} \,\, \leq \,\, CN^{-k}\ln^{k} N.
    \end{equation*}
    With this, the proof is concluded.
\end{proof}
In the following, we establish error equations to analyze the consistency error arising from WG-FEM.

\begin{lemma}\label{ch11_Z1_Z2}
    Assume that $u$ is the solution of \eqref{spbtp_md_problem}. Then, for any $ \psi = \{\psi_0, \psi_b\}\in \mathscr{S}_{h}^0$, the following holds  
    \begin{eqnarray}\label{Z1_Z2}
        -\varepsilon\left({u}_{xx},\psi_{0}\right) &=& \varepsilon\left(d_{w}(\mathscr{I}u), d_{w}\psi\right) - {Z}_{1}(u, \psi), \\[6pt]
         (a{u}_{x}, \psi_{0}) &=& (d_{w}^{a}(\mathscr{I}u), \psi_{0})- {Z}_{2}(u,\psi),
    \end{eqnarray}
    where
    \begin{eqnarray}\label{z1}
    {Z}_{1}(u, \psi) &=& \varepsilon\langle {u}_{x} - (\mathscr{I}u)_{x}, (\psi_0 - \psi_b)n\rangle\strut_{\partial \mathcal{T}_{N}}, \\[6pt]
    {Z}_{2}({u}, \psi) &=& \left({u} -\mathscr{I}u, (a\psi_0)_{x}\right)\strut_{\mathcal{T}_{N}}.\label{z2}
    \end{eqnarray}
    \end{lemma}

\begin{theorem}[Semi-discrete error equation]\label{semi_err_eqn}
    For a fixed time $t \in [0,T]$, let $e_h = \mathscr{I}u(t) - u_h(t) \in \mathscr{S}_{h}^{0}$. For any $v_h \in \mathscr{S}_{h}^{0}$, the following holds
    \begin{equation}\label{ch11_semi_discrete_err_eqn}
        ((e_h)_t, v_0) + A(e_h,v_h) = Z(u,v_h),
    \end{equation}
    where $Z(u,v_h) = Z_1(u,v_h) - Z_2(u,v_h) + Z_3(u, v_h)$ and $Z_3(u, v_h) = (c(\mathscr{I}u - u), v_{0})\strut_{\mathcal{T}_{N}}$.
\end{theorem}
\begin{proof}
    Testing our problem \eqref{spbtp_md_problem} by $v_h \in \mathscr{S}_{h}^{0}$, we get
    \begin{equation*}
        (u_t, v_0) - \varepsilon(u_{xx}, v_0) - (au_x, v_0) + (cu, v_0) = (f, v_0).
    \end{equation*}
    Now, we have
    \begin{equation*}
        (u_t, v_0) + \varepsilon(d_w(\mathscr{I}u), d_w v_h) - (d_w^a \mathscr{I}u, v_0) + (c\mathscr{I}u,v_0) - Z_1(u,v_h) + Z_2(u,v_h) -Z_3(u,v_h) = (f, v_0).
    \end{equation*}
    Adding $s_d(\mathscr{I}u, v_h)$ and $s_c(\mathscr{I}u, v_h)$ both sides and using the fact that $\mathscr{I}u$ is continuous, we get
    \begin{equation}\label{interp_semi_discrete}
        (\mathscr{I}u_t, v_0) + A(\mathscr{I}u, v_h) = (f, v_0) + Z_1(u, v_h) - Z_2(u, v_h) + Z_3(u,v_h).
    \end{equation}
    Now subtracting \eqref{ch11_semi_discrete} from \eqref{interp_semi_discrete}, we get
    \begin{equation}
        ((e_h)_t, v_0) + A(e_h,v_h) = Z(u,v_h).
    \end{equation}
    This completes the proof.
    % where $Z(u,v_h) = Z_1(u,v_h) - Z_2(u,v_h) + Z_3(u,v_h)$. This completes the proof.
\end{proof}
Now, we give an estimate for the semi-discrete scheme \eqref{ch11_semi_discrete} using Theorem \ref{semi_err_eqn}.

\begin{theorem}[Error estimate for semi-discrete formulation] 
    Consider the exact solution $u$ of \eqref{spbtp_md_problem} and the WG solution $u_h$ obtained from the semi-discrete formulation \eqref{ch11_semi_discrete}. Let $e_h(t) = \mathscr{I}u(t) - u_h(t)$, for a fixed $t \in [0,T]$. The following estimate can then be established
    \begin{equation}
        \|e_{h}(t)\|^2 \,+ \int_{0}^{t}\||e_h(s)\||^2ds \quad \leq \quad C\int_{0}^{t} (N^{-1}\ln N)^{2k}\,ds.
    \end{equation}
\end{theorem}
\begin{proof} Substitute $v_h = e_h$ in \eqref{ch11_semi_discrete_err_eqn} for every $t$ in $(0, T]$, we obtain
\[
(e_{h,t}, {e}_{h}) + A({e}_{h}, {e}_{h}) = {Z}(u, {e}_{h}).
\]
Invoking the coercivity of $A(\cdot, \cdot)$, we can express the above as
\[
\frac{1}{2}\dfrac{d}{dt}\|{e}_{h}\|^2 \,\, + \,\, \||{e}_{h}|\|^{2} \quad \leq \quad |{Z}(u, {e}_{h})|.
\]
Now in order to bound $Z(u, e_h)$, we will first bound $Z_1(u, e_h)$.

By invoking Lemma \ref{deri_interp_bound} and making use of the Cauchy-Schwarz inequality to obtain
\begin{eqnarray}
    |Z_1(u, e_h)| &\leq& \sum_{i=1}^{N}\varepsilon\|(u - \mathscr{I}u)^{\prime}\|\strut_{L^{2}(\partial M_{x_i})}\|e_0 - e_b\|\strut_{L^{2}(\partial M_{x_i})} \nonumber \\[2pt]
    &\leq& \Big(\sum_{i=1}^{N}\frac{\varepsilon^2}{\vartheta^i}\|(u - \mathscr{I}u)^{\prime}\|\strut_{L^{2}(\partial M_{x_i})}^2\Big)^{1/2}\Big(\sum_{i=1}^{N}\vartheta^i\|e_0 - e_b\|\strut_{L^{2}(\partial M_{x_i})}^{2}\Big)^{1/2} \nonumber \\[2pt]
    &\leq& C(N^{-1}\ln N)^{k}\||e_h\||. \label{Z1}
\end{eqnarray}
One can see that $Z_2(u, e_h) + Z_3(u, e_h) = (u - \mathscr{I}u, ae_0^{\prime}) + (u - \mathscr{I}u, (a^{\prime} - c)e_0)$. Using the Cauchy-Schwarz inequality and Lemma \ref{ch11interp_estimates}, we get
\begin{eqnarray}
    |(u - \mathscr{I}u, ae_0^{\prime})| &\leq& C\Big(\|u - \mathscr{I}u\|\strut_{L^\infty(\Omega_1)}\|e_{0}^{\prime}\|\strut_{L^1(\Omega_1)} + \|u - \mathscr{I}u\|\strut_{L^\infty(\Omega_2)}\|e_{0}^{\prime}\|\strut_{L^1(\Omega_2)}\Big) \nonumber \\[2pt]
    &\leq& C\Big((N^{-1}\ln N)^{k+1}\|e_{0}^{\prime}\|\strut_{L^1(\Omega_1)} + N^{-(k+1)}\|e_{0}^{\prime}\|\strut_{L^1(\Omega_2)}\Big). \nonumber
\end{eqnarray}
Using Cauchy-Schwarz inequality on $\Omega_1$ and inverse inequality on $\Omega_2$, we obtain
\begin{eqnarray*}
    \|e_{0}^{\prime}\|\strut_{L^{1}(\Omega_1)} &\leq& C(\ln N)^{1/2}\||e_h\||, \\[2pt]
    \|e_{0}^{\prime}\|\strut_{L^{1}(\Omega_2)} &\leq& CN\||e_h\||.
\end{eqnarray*}
Hence we have 
\begin{eqnarray}
    |(u - \mathscr{I}u, ae_0^{\prime})| &\leq& C\Big((N^{-1}\ln N)^{k}\cdot N^{-1}\ln^{3/2} N + N^{-k} \Big)\||e_h\|| \nonumber \\[2pt]
    &\leq& C(N^{-1}\ln N)^{k}\||e_h\||, \label{Z2}
\end{eqnarray}
where we have used $N^{-1}\ln^{3/2} N < 1$. Similarly, we can have
\begin{eqnarray}
    |(u - \mathscr{I}u, (a^{\prime} - c)e_0)| &\leq& \|u - \mathscr{I}u\|\strut_{L^2(\Omega)}\|e_0\|\strut_{L^2(\Omega)} \nonumber \\[2pt]
    &\leq& CN^{-(k+1)}\||e_h\||. \label{Z3}
\end{eqnarray}
By using \eqref{Z1}-\eqref{Z3}, we get 
\[
\frac{1}{2}\dfrac{d}{dt}\|{e}_{h}\|^2 \,\, + \,\, \||{e}_{h}|\|^{2} \quad \leq \quad C(N^{-1}\ln N)^{k}\||e_h\||.
\]
Now, integrating from $0$ to $t$, we get
\begin{eqnarray*}
    \frac{1}{2}\|{e}_{h}(t)\|^2 \,\, + \,\, \int_{0}^{t}\||{e}_{h}(s)|\|^{2}ds \quad \leq \quad C\int_{0}^{t}(N^{-1}\ln N)^{k}\||e_h\||,
\end{eqnarray*}
which can be simplified to get
\begin{eqnarray*}
    \|{e}_{h}(t)\|^2 \,\, + \,\, \int_{0}^{t}\||{e}_{h}(s)|\|^{2}ds \quad \leq \quad C\int_{0}^{t}(N^{-1}\ln N)^{2k}ds.
\end{eqnarray*}
Hence the proof is now completed.
\end{proof}

\begin{theorem}[Fully-discrete error equation]
    Let $e_h^{n} = \mathscr{I}u(t_n) - u_h(t_n) \in \mathscr{S}_{h}^{0}$. For any $v_h \in \mathscr{S}_{h}^{0}$, the following holds
    \begin{eqnarray}\label{ch11_fully_discrete_err_eqn}
        (\overline{\partial} {e}_h^n, v_0) + A(\theta e_{h}^{n} + (1-\theta)e_{h}^{n-1}, v_{h}) &=& (\overline{\partial}u^{n}-(\theta u_{t}(t_n) + (1-\theta)u_{t}(t_{n-1})), v_0) \nonumber \\[2pt]
        && + Z(\theta u(t_n) + (1-\theta)u(t_{n-1}),v_h).
    \end{eqnarray}
\end{theorem}
\begin{proof}
    We know from Theorem \ref{semi_err_eqn} that
    \[
    (\mathscr{I}u_t, v_0) + A(\mathscr{I}u, v_h) = (f, v_0) + Z(u, v_h),
    \]
    which leads us to 
    \begin{eqnarray}
    A(\theta\mathscr{I}u(t_n) + (1-\theta)\mathscr{I}u(t_{n-1}), v_h) &=& -(\theta \mathscr{I}u_{t}(t_n) + (1-\theta)\mathscr{I}u_{t}(t_{n-1}), v_h)\nonumber \\[2pt]
    && + (\theta f(t_n) + (1-\theta)f(t_{n-1}), v_{0}) \nonumber \\[2pt]
    && + Z(\theta u(t_n) + (1-\theta)u(t_{n-1}), v_h). \label{full_disc_sub}
    \end{eqnarray}
    On subtracting \eqref{full_disc_sub} from \eqref{fully_discrete_form}, we have
    \begin{eqnarray*}
        -(\overline{\partial} u_h^{n}, v_0) + A(\theta e_{h}^{n} + (1-\theta)e_{h}^{n-1}, v_{h}) &=&  -(\theta \mathscr{I}u_{t}(t_n) + (1-\theta)\mathscr{I}u_{t}(t_{n-1}), v_h) \\[2pt]
        && + Z(\theta u(t_n) + (1-\theta)u(t_{n-1}), v_h).
    \end{eqnarray*}
    Adding $(\overline{\partial} \mathscr{I}u^{n}, v_0)$ on both sides, we get
    \begin{eqnarray*}
        (\overline{\partial} e_h^{n}, v_0) + A(\theta e_{h}^{n} + (1-\theta)e_{h}^{n-1}, v_{h}) &=&  (\overline{\partial} \mathscr{I}u^{n}-(\theta \mathscr{I}u_{t}(t_n) + (1-\theta)\mathscr{I}u_{t}(t_{n-1})), v_0) \\[2pt]
        && + Z(\theta u(t_n) + (1-\theta)u(t_{n-1}), v_h) \\[4pt]
        &=& (\overline{\partial}u^{n}-(\theta u_{t}(t_n) + (1-\theta)u_{t}(t_{n-1})), v_0) \\[2pt]
        && + Z(\theta u(t_n) + (1-\theta)u(t_{n-1}), v_h),
    \end{eqnarray*}
    which is the required fully-discrete error equation.
\end{proof}
We now present the following theorem, which provides an error estimate for the fully discrete scheme given in \eqref{fully_discrete_form}.

\begin{theorem}\label{main_theorem}
Let $u \in H^{k+1}(\Omega)$ be the solution of problem \eqref{spbtp_md_problem} and $u_h^n$ be the solution of the fully-discrete scheme \eqref{fully_discrete_form}. Then the following bound holds

For $\theta \in (1/2,1]$
\begin{eqnarray}\label{ch11_bound_full_discrete_1}
    \|e_h^n\|^2 \,\, + \,\, \zeta \sum_{j=1}^{n}\||\theta e_{h}^{j} + (1-\theta)e_{h}^{j-1}\||^2 &\leq& C\zeta^2\int_{0}^{t_n}\|u_{tt}(s)\|^2\,ds   + CN^{-2k}\ln^{2k} N, \nonumber \\
\end{eqnarray}
and for $\theta = 1/2$
\begin{eqnarray}\label{ch11_bound_full_discrete_2}
    \|e_h^n\|^2 \,\, + \,\, \zeta \sum_{j=1}^{n}\||\theta e_{h}^{j} + (1-\theta)e_{h}^{j-1}\||^2 &\leq& C\zeta^4\int_{0}^{t_n}\|u_{ttt}(s)\|^2\,ds   + CN^{-2k}\ln^{2k} N. \nonumber \\
\end{eqnarray} 
\end{theorem}

\begin{proof}
    Substituting $v_h = \theta e_{h}^{n} + (1 - \theta)e_{h}^{n-1}$ in \eqref{ch11_fully_discrete_err_eqn}, we get
    \begin{eqnarray*}
        (\overline{\partial} {e}_h^n, \theta e_{h}^{n} + (1 - \theta)e_{h}^{n-1}) &+& A(\theta e_{h}^{n} + (1-\theta)e_{h}^{n-1}, \theta e_{h}^{n} + (1 - \theta)e_{h}^{n-1}) \\[4pt]
        && = (\overline{\partial}u(t_n) - (\theta u_{t}(t_n) + (1-\theta)u_t(t_{n-1})), \theta e_{h}^{n} + (1 - \theta)e_{h}^{n-1}) \\[4pt]
        && \quad + Z(\theta u(t_n) + (1-\theta)u(t_{n-1}), \theta e_{h}^{n} + (1 - \theta)e_{h}^{n-1}).
    \end{eqnarray*}
    One can write
    \begin{eqnarray*}
        (\overline{\partial} {e}_h^n, \theta e_{h}^{n} + (1 - \theta)e_{h}^{n-1}) &=& \dfrac{1}{\zeta}(\theta - \dfrac{1}{2})\|e_{h}^{n} - e_{h}^{n-1}\|^2 + \dfrac{1}{2\zeta}\|e_{h}^{n}\|^2 - \dfrac{1}{2\zeta}\|e_{h}^{n-1}\|^2.
    \end{eqnarray*}
    Now by using the coercivity of $A(\cdot , \cdot)$  and the Cauchy-Schwarz inequality, we get
    \begin{eqnarray*}
        \dfrac{1}{\zeta}(\theta - \dfrac{1}{2})\|e_{h}^{n} - e_{h}^{n-1}\|^2 &+& \dfrac{1}{2\zeta}\|e_{h}^{n}\|^2 - \dfrac{1}{2\zeta}\|e_{h}^{n-1}\|^2 + \||\theta e_{h}^{n} + (1 - \theta)e_{h}^{n-1}\||^2 \\[4pt]
        &\leq& \|\overline{\partial}u^{n}-(\theta u_{t}(t_n) + (1-\theta)u_{t}(t_{n-1}))\|\cdot \||\theta e_{h}^{n} + (1 - \theta)e_{h}^{n-1}\|| \\[4pt]
        && +  |Z(\theta u(t_n) + (1-\theta)u(t_{n-1}), \theta e_{h}^{n} + (1 - \theta)e_{h}^{n-1})|.
    \end{eqnarray*}
    For $\theta \in (1/2,1]$ we have:
    \begin{eqnarray*}
        \|\overline{\partial}u^{n}-(\theta u_{t}(t_n) + (1-\theta)u_{t}(t_{n-1}))\|^2 &=& \int_{\Omega}\dfrac{1}{\zeta^2}\Big(\int_{t_{n-1}}^{t_n}(s - (1 - \theta)t_n - \theta t_{n-1})u_{ss}\,ds\Big)^2dx \\[4pt]
        &\leq& \int_{\Omega}\dfrac{1}{\zeta^2}\Big(\int_{t_{n-1}}^{t_n}(s - (1 - \theta)t_n - \theta t_{n-1})^2\,ds\Big)\Big(\int_{t_{n-1}}^{t_n}u_{ss}^{2}\,ds\Big)dx \\[4pt]
        &\leq& \int_{\Omega}\theta^2\zeta \int_{t_{n-1}}^{t_n}u_{ss}^{2}\,dsdx \\[4pt]
        &\leq& \zeta\int_{t_{n-1}}^{t_n}\|u_{ss}\|^{2}\,ds.
    \end{eqnarray*}
    Also, we have the bound
    \begin{eqnarray*}
        |Z(\theta u(t_n) + (1-\theta)u(t_{n-1}), \theta e_{h}^{n} + (1 - \theta)e_{h}^{n-1})| &\leq& C(N^{-1}\ln N)^{k}\||\theta e_{h}^{n} + (1 - \theta)e_{h}^{n-1}\||.
    \end{eqnarray*}
    Therefore, we have
    \begin{eqnarray*}
         \dfrac{1}{\zeta}(\theta - \dfrac{1}{2})\|e_{h}^{n} - e_{h}^{n-1}\|^2 &+& \dfrac{1}{2\zeta}\|e_{h}^{n}\|^2 - \dfrac{1}{2\zeta}\|e_{h}^{n-1}\|^2 + \||\theta e_{h}^{n} + (1 - \theta)e_{h}^{n-1}\||^2 \\[4pt]
         &\leq& \zeta^{1/2}\Big(\int_{t_{n-1}}^{t_n}\|u_{ss}\|^{2}\,ds\Big)^{1/2}\||\theta e_{h}^{n} + (1 - \theta)e_{h}^{n-1}\|| \\[4pt]
         &&+ C(N^{-1}\ln N)^{k}\||\theta e_{h}^{n} + (1 - \theta)e_{h}^{n-1}\|| \\[4pt]
         &\leq& {\zeta}\Big(\int_{t_{n-1}}^{t_n}\|u_{ss}\|^{2}\,ds\Big) + \dfrac{1}{4}\||\theta e_{h}^{n} + (1 - \theta)e_{h}^{n-1}\||^2 \\[4pt]
         && + \dfrac{1}{4}\||\theta e_{h}^{n} + (1 - \theta)e_{h}^{n-1}\||^2 + CN^{-2k}\ln^{2k} N.
    \end{eqnarray*}
    On simplification, we have
    \begin{eqnarray*}
           \dfrac{1}{2\zeta}\|e_{h}^{n}\|^2 - \dfrac{1}{2\zeta}\|e_{h}^{n-1}\|^2 + \dfrac{1}{2}\||\theta e_{h}^{n} + (1 - \theta)e_{h}^{n-1}\||^2 &\leq& {\zeta}\Big(\int_{t_{n-1}}^{t_n}\|u_{ss}\|^{2}\,ds\Big) + CN^{-2k}\ln^{2k} N,
    \end{eqnarray*}
    or
    \begin{eqnarray*}
           \|e_{h}^{n}\|^2 - \|e_{h}^{n-1}\|^2 + \zeta \||\theta e_{h}^{n} + (1 - \theta)e_{h}^{n-1}\||^2 &\leq& C\zeta^2\int_{t_{n-1}}^{t_n}\|u_{ss}\|^{2}\,ds + C\zeta N^{-2k}\ln^{2k} N.
    \end{eqnarray*}
    Now, using the fact that $e_h^{0} = \mathscr{I}u(0) - u_{h}(0) = 0$ and summing over $j$ from 1 upto $n$, we get
    \begin{eqnarray*}
        \|e_{h}^{n}\|^2 + \zeta \sum_{j=1}^{n}\||\theta e_{h}^{j} + (1 - \theta)e_{h}^{j-1}\||^2 &\leq& C\zeta^2\int_{0}^{t_n}\|u_{ss}\|^{2}\,ds + C\sum_{j=1}^{n}\zeta(N^{-1}\ln N)^{2k} \\[4pt]
        &\leq&  C\zeta^2\int_{0}^{t_n}\|u_{ss}\|^{2}\,ds + CN^{-2k}\ln^{2k} N. 
    \end{eqnarray*}
    For the case $\theta = 1/2$, it follows that
    \begin{eqnarray*}
        \|\overline{\partial}u^{n}-(\theta u_{t}(t_n) + (1-\theta)u_{t}(t_{n-1})\|^2 &=& \int_{\Omega}\dfrac{1}{4\zeta^2}\Big(\int_{t_{n-1}}^{t_n}(s - t_n)(s - t_{n-1})u_{sss}\,ds\Big)^2dx \\[4pt]
        &\leq& \int_{\Omega}\dfrac{1}{4\zeta^2}\Big(\int_{t_{n-1}}^{t_n}(s - t_n)^2(s - t_{n-1})^2\,ds\Big)\Big(\int_{t_{n-1}}^{t_n}u_{sss}^{2}\,ds\Big)dx \\[4pt]
        &\leq& \dfrac{1}{120}\zeta^3\int_{t_{n-1}}^{t_n}\|u_{sss}\|^{2}\,ds.
    \end{eqnarray*}
    Following on the same lines as we proved \eqref{ch11_bound_full_discrete_1}, we have
    \begin{eqnarray}
    \|e_h^n\|^2 \,\, + \,\, \zeta \sum_{j=1}^{n}\||\theta e_{h}^{j} + (1-\theta)e_{h}^{j-1}\||^2 &\leq& C\zeta^4\int_{0}^{t_n}\|u_{ttt}(s)\|^2\,ds   + CN^{-2k}\ln^{2k} N. \nonumber \\
\end{eqnarray}
This concludes the proof.
\end{proof}

\section{Numerical Experiments}\label{Numerical}
To demonstrate the effectiveness and accuracy of the proposed method, we present numerical results that both support and validate the theoretical predictions. We discretize in time via the backward Euler (BE, $\theta = 1$) and Crank-Nicolson (CN, $\theta = 1/2$) methods. Let $e_h^n = u^n - u_h^n$ be the error. To compute the convergence order, we apply the formula
\[
\mbox{Order} = \dfrac{\ln(e_{h}^{n}/e_{2h}^{n})}{\ln(2\ln(N)/\ln(2N))}.
\]
\begin{example}
Consider the following problem
\begin{equation}\label{ch11_ex1}
\left\{\begin{array}{ll}
u_{t} -\varepsilon u_{xx} -x^{q}u_{x} + u = f(x,t), \, (x,t) \in \,\, (0,1) \times (0,1),\\[6pt]
u(x,0) = 0, \,\, x \in [0,1], \\[6pt]
u(0,t) = 0, u(1,t) = 0, \,\, t \in [0,1].
\end{array}\right.
\end{equation}
\end{example}
To obtain the exact solution
\[
u(x,t) = (1 - e^{-t})\Big(1 - x + xe^{-1/\sqrt{\varepsilon}} - e^{-x/\sqrt{\varepsilon}}\Big),
\]
the source term $f(x,t)$ is appropriately chosen.

We obtain the energy norm error and the corresponding order of convergence of our proposed scheme for various values of $\varepsilon, N$ and $k$, which are displayed in Tables \ref{BE_table_ex1} and \ref{CN_table_ex1}. For spatial error to dominate we fixed $\zeta = 1/5000$. These tables illustrate that the energy norm convergence rate is of order $\mathcal{O}((N^{-1}\ln N)^{k})$, which closely aligns with the Theorem \ref{main_theorem}. Figures \ref{BE_fig_ex1} and \ref{CN_fig_ex1} each present two subfigures: a surface plot of the numerical solution and a comparison of the exact and numerical solution at the final time $T=1$, corresponding to $\theta = 1$ and $\theta = 1/2$, respectively. To validate the theoretical convergence rates, we plot the energy norm errors with respect to mesh size $N$ on a logarithmic scale, for varying perturbation parameters $\varepsilon$ and polynomial degrees $k = 1, 2$. Figure \ref{BE_fig_ex2} depicts these results for $\theta = 1$, while Figure \ref{CN_fig_ex2} illustrates the corresponding behavior for $\theta = 1/2$. Both figures highlight the alignment of numerical errors with the predicted $\mathcal{O}((N^{-1} \ln N)^k)$ rates, demonstrating the schemes' robustness across parameter regimes. 
%%%%%%%%%%%%%%%%%%%%%%%%%%%%%%%

\begin{table}[ht]
\centering
\caption{Computed energy norm errors for Example 1 with $q=1, \theta = 1$ and $\zeta = 1/5000$.}\label{BE_table_ex1}
\begin{tabular}{lllllllll}
\toprule
&&\makebox[3em]{$\varepsilon = 10^{-4}$}&\makebox[3em]{Order}&\makebox[3em]{$\varepsilon = 10^{-5}$}&\makebox[3em]{Order}&{$\varepsilon = 10^{-6}$}&{Order}\\
\midrule
&{$N$}&\makebox[3em]{$\||\cdot\||$}&\makebox[3em]{}&\makebox[3em]{$\||\cdot\||$}&\makebox[3em]{}&{$\||\cdot\||$}&{}\\
\midrule
&$8$ & 1.9078e-04 & -- & 1.0671e-04 & -- &5.9903e-05& --\\
&&&&&&&\\
&$16$&  1.2728e-04 &0.9981& 7.1196e-05  &0.9980 &3.9970e-05&0.9978\\
&&&&&&&\\
$k=1$&$32$ & 7.9599e-05   &0.9988&  4.4532e-05 &0.9983 &2.5001e-05&0.9983\\
&&&&&&&\\
&$64$ & 4.7775e-05    &0.9994&  2.6730e-05 & 0.9992&1.5008e-05&0.9991\\
&&&&&&&\\
&$128$ & 2.7873e-05   &0.9997&  1.5596e-05 &0.9996 &8.7581e-06&0.9993\\

\midrule
&$8$ &5.1432e-05& -- & 2.8735e-05& -- &1.6126e-05& --\\
&&&&&&&\\
&$16$ &2.4257e-05&1.8535&  1.3557e-05  &1.8527&7.6107e-06&1.8518\\
&&&&&&&\\
$k=2$&$32$ &9.8014e-06& 1.9281 &  5.4824e-06  & 1.9263&3.0830e-06&1.9226\\
&&&&&&&\\
&$64$ &3.5875e-06&1.9675&  2.0149e-06  &1.9595&1.1470e-06&1.9356\\
&&&&&&&\\
&$128$ & 1.2462e-06&1.9618&  7.2259e-07  & 1.9025&4.4848e-07&1.7422\\
\bottomrule
\end{tabular}
\end{table}

\begin{figure}[!ht]
\centerline{%
	\begin{tabular}{cc}
		\resizebox*{8cm}{!}{\includegraphics{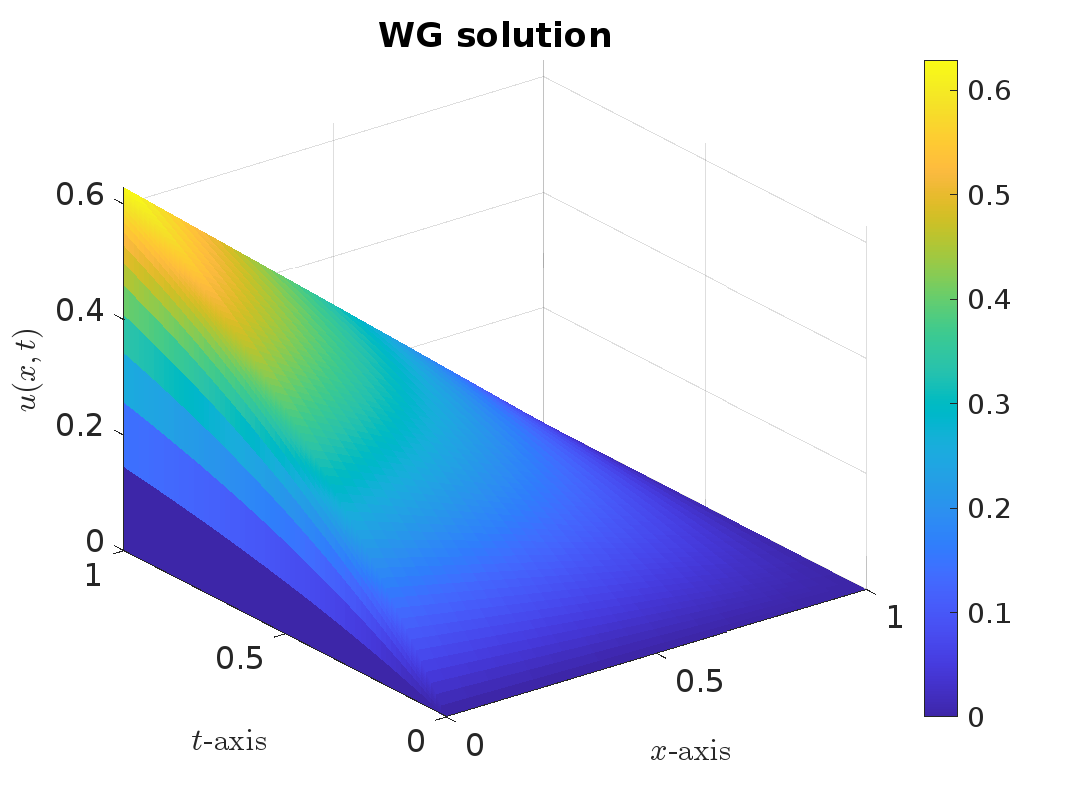}}
		&
		\resizebox*{8cm}{!}{\includegraphics{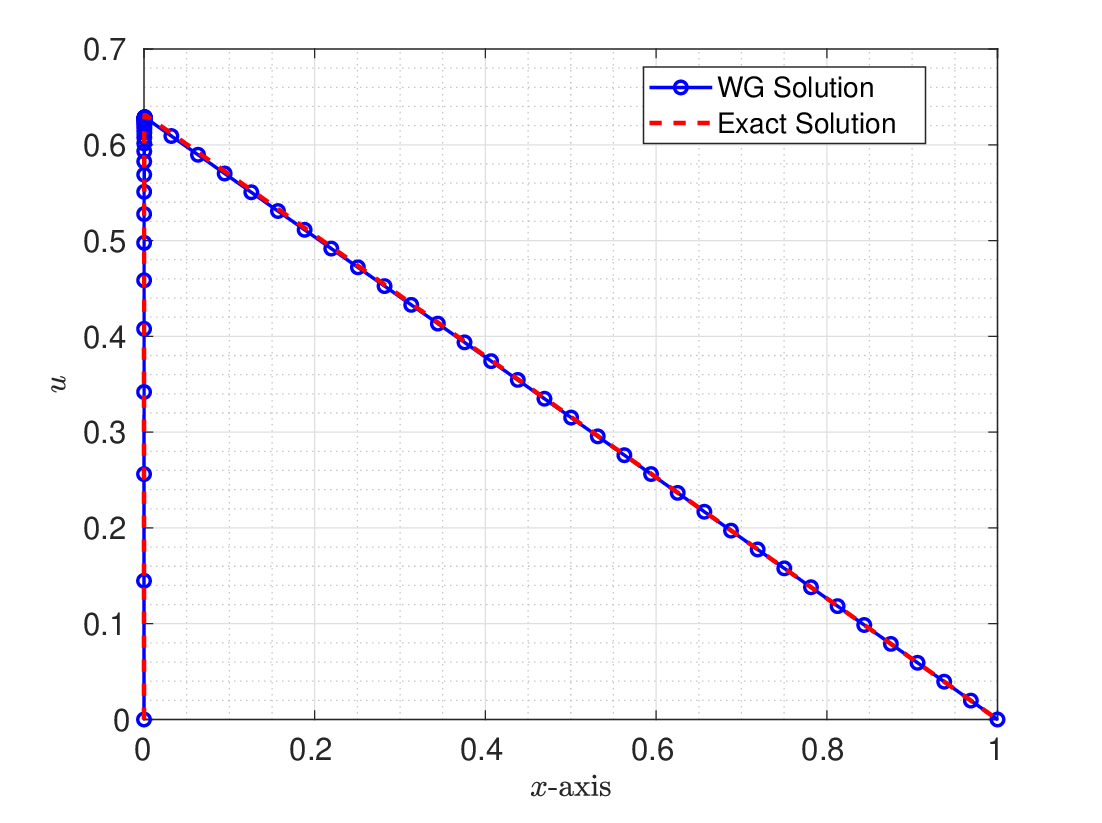}}\\
		{\it {\em (a) BE-WG surface plot.}} & {\it {\em (b) BE-WG and exact solution at $T=1$}.}
	\end{tabular}}
    \caption{Comparison of numerical solution obtained using the backward Euler (BE) time-stepping method for $\varepsilon = 10^{-8}$, with a spatial discretization of $N = 64$.
} 
    \label{BE_fig_ex1}
\end{figure}

\begin{figure}[!ht]
\centerline{%
	\begin{tabular}{cc}
		\resizebox*{8cm}{!}{\includegraphics{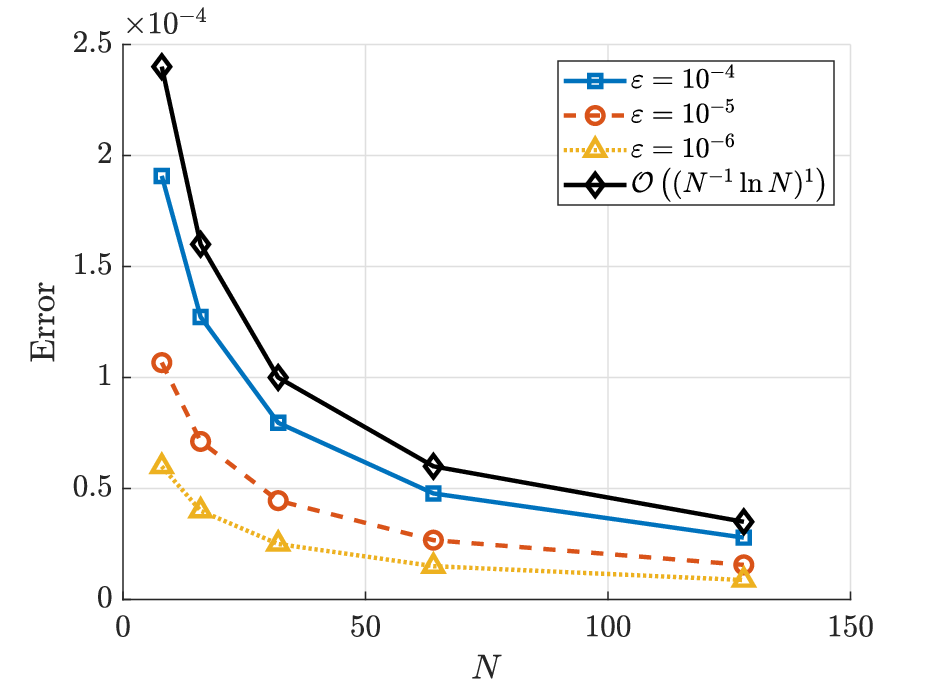}}
		&
		\resizebox*{8cm}{!}{\includegraphics{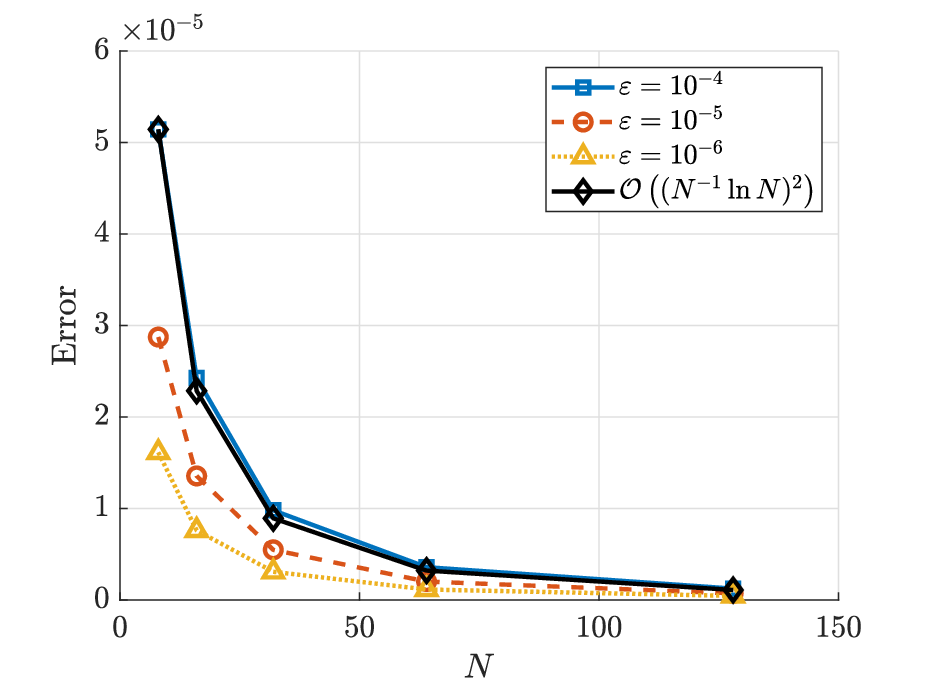}}\\
		{\it {\em (a) Using $\mathbb{P}_1$ element. }} & {\it {\em (b) Using $\mathbb{P}_2$ element.}}
	\end{tabular}
} \caption{Loglog plot of $N$ vs energy norm for $\theta=1$ using $k=1$ and 2.} \label{BE_fig_ex2}
\end{figure}

\begin{table}[!ht]
\centering
\caption{Computed energy norm errors for Example 1 with $q=1, \theta = 1/2$ and $\zeta = 1/5000$.}\label{CN_table_ex1}
\begin{tabular}{lllllllll}
\toprule
&&\makebox[3em]{$\varepsilon = 10^{-4}$}&\makebox[3em]{Order}&{$\varepsilon = 10^{-5}$}&{Order}&\makebox[3em]{$\varepsilon = 10^{-6}$}&\makebox[3em]{Order}\\
\midrule
&{$N$}&\makebox[3em]{$\||\cdot\||$}&\makebox[3em]{}&{$\||\cdot\||$}&{}&\makebox[3em]{$\||\cdot\||$}&\makebox[3em]{}\\
\midrule
&$8$ & 1.9079e-04 & -- &   1.0671e-04 & -- &  5.9905e-05& --\\
&&& &&  &&\\
&$16$ &  1.2729e-04 & 0.9981 &    7.1197e-05 & 0.9980 &3.9971e-05 & 0.9979\\
&&& &&  &&\\
$k=1$&$32$ &  7.9601e-05 & 0.9988 &   4.4532e-05 & 0.9984 &  2.5001e-05& 0.9983\\
&&& &&  &&\\
&$64$ &  4.7775e-05 & 0.9994 &   2.6730e-05  & 0.9992 & 1.5007e-05 & 0.9992\\
&&& &&  &&\\
&$128$ & 2.7873e-05 & 0.9997 & 1.5595e-05  & 0.9997 & 8.7553e-06 & 0.9997\\
\midrule
&$8$ & 5.1431e-05 & -- & 2.9488e-05 & -- &  1.6125e-05& --\\
&&& &&  &&\\
&$16$ &  2.4256e-05 & 1.8536 & 1.3726e-05  & 1.8858 & 7.6074e-06& 1.8528\\
&&& &&  &&\\
$k=2$&$32$ & 9.7984e-06 & 1.9286 & 5.5056e-06  &1.9438  & 3.0744e-06& 1.9277\\
&&& &&  &&\\
&$64$ & 3.5800e-06 & 1.9710 & 2.0054e-06  & 1.9771 & 1.1236e-06& 1.9706\\
&&& &&  &&\\
&$128$ & 1.2251e-06 & 1.9895 & 6.8546e-07  & 1.9917 & 3.8452e-07& 1.9893\\
\bottomrule
\end{tabular}
\end{table}

\begin{figure}[!ht]
\centerline{%
	\begin{tabular}{cc}
		\resizebox*{8cm}{!}{\includegraphics{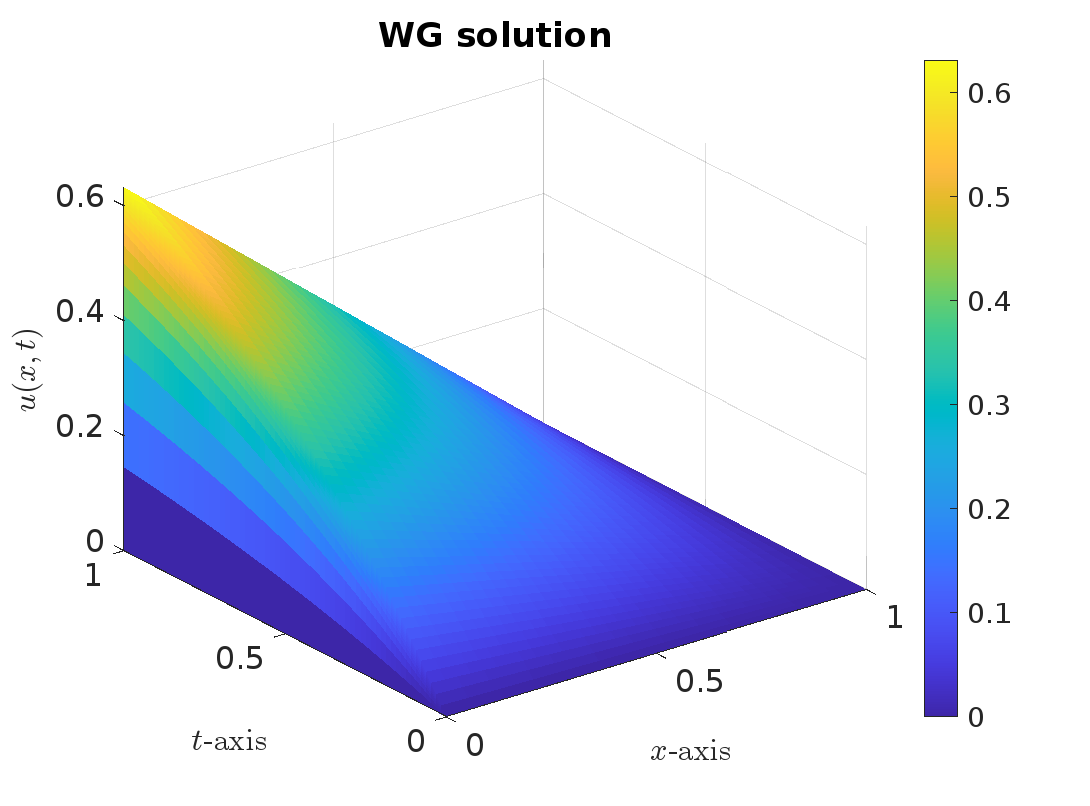}}
		&
		\resizebox*{8cm}{!}{\includegraphics{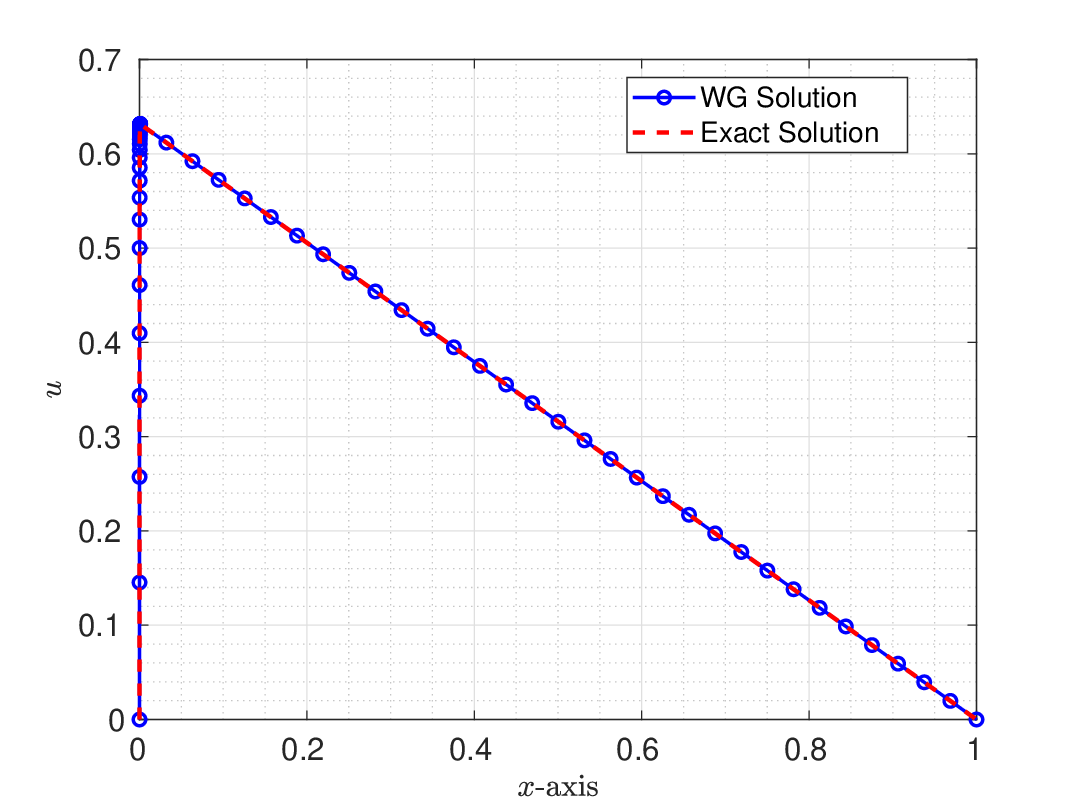}}\\
		{\it {\em (a) CN-WG surface plot.}} & {\it {\em (b) CN-WG and exact solution at $T=1$}.}
	\end{tabular}
} \caption{Comparison of numerical solution obtained using the Crank-Nicolson (CN) time-stepping method for $\varepsilon = 10^{-8}$, with a spatial discretization of $N = 64$.}\label{CN_fig_ex1}
\end{figure}

\begin{figure}[!ht]
\centerline{%
	\begin{tabular}{cc}
		\resizebox*{8cm}{!}{\includegraphics{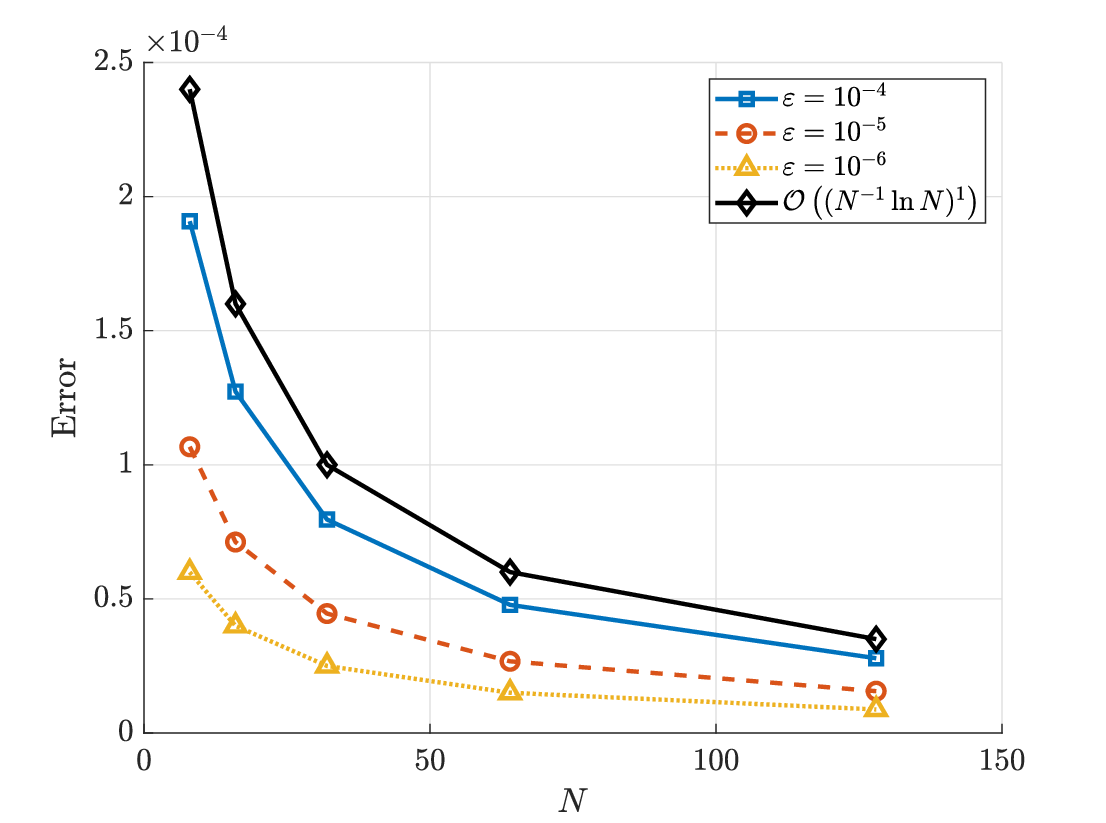}}
		&
		\resizebox*{8cm}{!}{\includegraphics{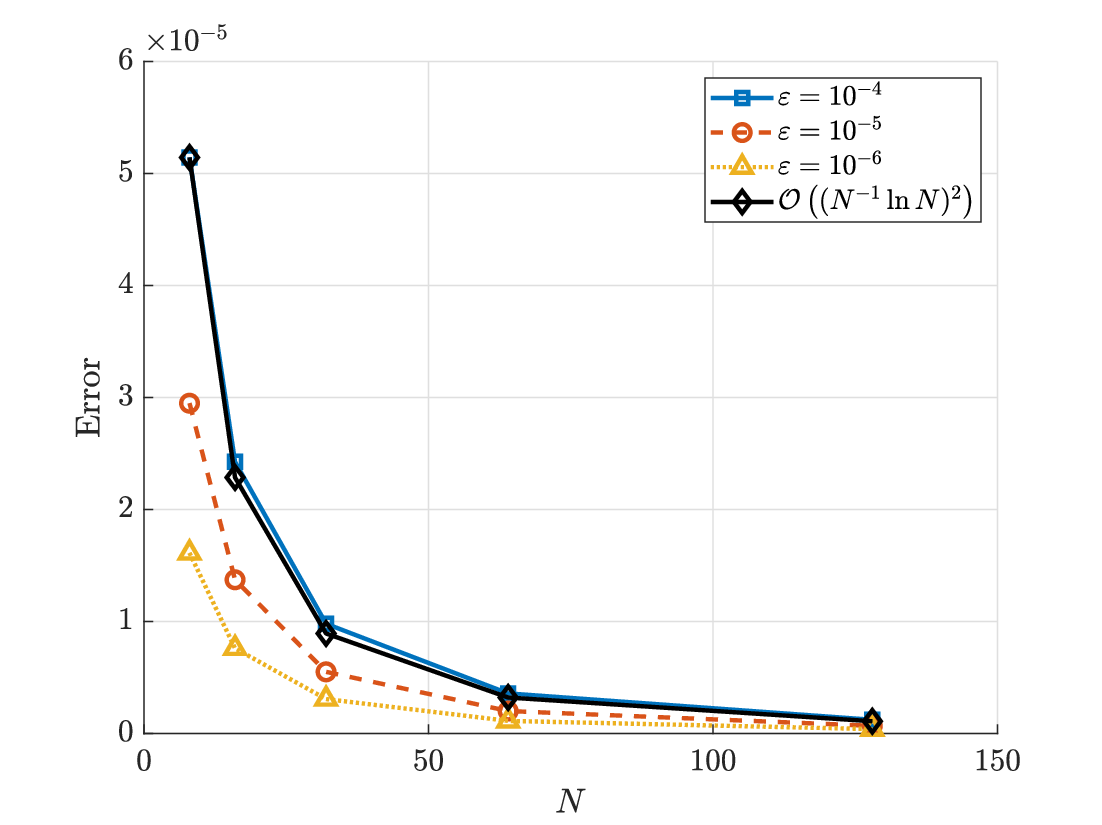}}\\
		{\it {\em (a) Using $\mathbb{P}_1$ element. }} & {\it {\em (b) Using $\mathbb{P}_2$ element.}}
	\end{tabular}
} \caption{Loglog plot of $N$ vs energy norm for $\theta=1/2$ using $k=1$ and 2.} \label{CN_fig_ex2}
\end{figure}

To highlight the differences in convergence behavior, we fixed a relatively coarse temporal step size ($\zeta = 1/500$) and evaluated the errors for $\theta = 1$ and $\theta = 1/2$. As presented in Table \ref{BE_CN_comp_500}, the scheme with $\theta = 1$ shows deteriorating convergence rates with increasing $N$, due to the dominant temporal discretization error. In contrast, the scheme with $\theta = 1/2$ retains consistent convergence, owing to its higher temporal accuracy.

To demonstrate how the parameter $q$ affects the solution of the degenerate PDE, we evaluate the errors in the energy norm along with the associated convergence rates
for different values of $q$, keeping $\varepsilon = 10^{-8}$. The results, presented in Table \ref{q_error_comp}, indicate that the error bound is independent of $q$. In Table \ref{comparison}, the maximum norm errors are compared with the values reported in \cite{majumdar2017second}.

\begin{table}[!ht]
\caption{Energy norm error comparison for $\varepsilon = 10^{-8}$ with $\zeta = 1/500$ using linear elements ($k = 1$).}
\centering
\begin{tabular}{lccccc}
\toprule
Method & \multicolumn{5}{c}{$N$} \\
\cmidrule(lr){2-6}
 & 32 & 64 & 128 & 256 & 512\\
\midrule
BE-WG scheme & 2.6024e-05 & 1.6678e-05 & 1.1397e-05 & 8.8526e-06 & 7.8287e-06\\
$(\theta = 1)$ & -- & 0.9461 & 0.8710 & 0.7064 & 0.4515 \\ 
CN-WG scheme & 2.4985e-05 & 1.4997e-05 & 8.7494e-06 & 4.9999e-06 & 2.8125e-06\\
$(\theta = 1/2)$ & -- & 0.9984 & 0.9992 & 0.9997 & 0.9999\\
\bottomrule
\end{tabular}
\label{BE_CN_comp_500}
\end{table}

\begin{table}[!ht]
\caption{Error and convergence rates in the energy norm for various $q$ and $N$ values with $\zeta = 1/5000$, $k=1$, and $\varepsilon = 10^{-8}$.}
\centering
\begin{tabular}{lccccc}
\toprule
Method & \multicolumn{5}{c}{$N$} \\
\cmidrule(lr){2-6}
 & 8 & 16 & 32 & 64 & 128 \\
\midrule
\multicolumn{6}{l}{BE-WG scheme ($\theta = 1$)} \\
$q=2$ & 1.8814e-05 & 1.2602e-05 & 7.8978e-06 & 4.7483e-06 & 2.7795e-06 \\
 & -- & 0.9884 & 0.9942 & 0.9960 & 0.9936 \\
$q=3$ & 1.8815e-05 & 1.2602e-05 & 7.8982e-06 & 4.7491e-06 & 2.7808e-06 \\
& -- & 0.9884 & 0.9941 & 0.9958 & 0.9930 \\
\addlinespace
\multicolumn{6}{l}{CN-WG scheme ($\theta = 1/2$)} \\
$q=2$ & 1.8813e-05 & 1.2599e-05 & 7.8932e-06 & 4.7407e-06 & 2.7665e-06 \\
& -- & 0.9887 & 0.9950 & 0.9980 & 0.9993 \\
$q=3$ & 1.8813e-05 & 1.2599e-05 & 7.8932e-06 & 4.7407e-06 & 2.7665e-06 \\
& -- & 0.9887 & 0.9950 & 0.9980 & 0.9993 \\
\bottomrule
\end{tabular}
\label{q_error_comp}
\end{table}

\begin{table}[!ht]
\centering
\caption{Maximum norm error comparison for $\mathcal{M}=N$ with $k=1$ and $\varepsilon = 2^{-10}$.}
\label{comparison}
\begin{tabular}{l c c c c c}
\toprule
Method & \multicolumn{5}{c}{$N$} \\
\cmidrule(lr){2-6}
 & 32 & 64 & 128 & 256 & 512 \\
\midrule
CN-WG scheme ($\theta = 1/2$) & 3.1944e-03 & 1.1206e-03 & 3.7780e-04 & 1.2336e-04 & 3.8995e-05 \\
& -- & 2.0507 & 2.0172 & 2.0000 & 2.0016 \\
\addlinespace
Method in \cite{majumdar2017second} & 1.3095e-02 & 8.7805e-03 & 5.4450e-03 & 3.2401e-03 & 1.8737e-03 \\
& -- & 0.5767 & 0.6894 & 0.7489 & 0.7901 \\
\bottomrule
\end{tabular}
\end{table}

\section{Conclusion and future direction}\label{conclusion}
In conclusion, we have developed and rigorously analyzed a WG-FEM for SPBTPPs, employing an implicit $\theta$-scheme for temporal discretization. The proposed scheme exhibits stability as well as uniform convergence with respect to the energy norm, as validated by comprehensive numerical experiments. As part of ongoing work, we are focusing on formulating and analyzing ADI-type operator splitting schemes within the WG-FEM framework for multidimensional SPBTPPs.

\subsection*{Declarations}
\subsection*{Funding}
The first author gratefully acknowledges the financial support provided by the Indian Institute of Technology Guwahati.

\subsection*{Data Availability}
Data sharing is not applicable to this paper because no datasets were created or
examined during the current study.
\subsection*{Conflict of interest}
The authors declare that they have no conflict of interest.

% \subsection*{Ethical approval}
% This article does not contain any studies with human participants or animals performed by any of the authors.


\end{document}